\documentclass[11pt,reqno]{amsart}
\usepackage{amssymb,mathrsfs,graphicx,enumerate}
\usepackage{amsmath,amsfonts,amssymb,amscd,amsthm,bbm}

\usepackage[retainorgcmds]{IEEEtrantools}
\usepackage{kotex}

\makeatletter
\@namedef{subjclassname@2020}{\textup{2020} Mathematics Subject Classification}
\makeatother

\usepackage{algorithm}
\usepackage{algpseudocode}
\usepackage{tabularx} 
\usepackage{booktabs} 

\usepackage{colortbl}
\usepackage{cite}
\usepackage{subcaption}
\usepackage{nccmath}
\usepackage{bm}
\usepackage{upgreek}
\title[Discrete CBO]{Discrete Consensus-Based Optimization}

\author[Byeon]{Junhyeok Byeon}
\address[Junhyeok Byeon]{\newline Research Institute of Basic Sciences\newline Seoul National University, Seoul 08826, Republic of Korea}
\email{giugi2486@snu.ac.kr}

\author[Ha]{Seung-Yeal Ha}
\address[Seung-Yeal Ha]{\newline Department of Mathematical Sciences\newline Seoul National University, Seoul 08826, Republic of Korea}
\email{syha@snu.ac.kr}

\author[Won]{Joong-Ho Won}
\address[Joong-Ho Won]{\newline Department of Statistics \newline Seoul National University, Seoul 08826, Republic of Korea}
\email{wonj@stats.snu.ac.kr}

\DeclareMathOperator*{\argmin}{arg\,min}
\DeclareMathOperator*{\essinf}{ess\,inf}

\newtheorem{theorem}{Theorem}[section]

\newtheorem{proposition}{Proposition}[section]
\newtheorem{remark}{Remark}[section]
\newtheorem{example}{Example}[section]

\newtheorem{definition}{Definition}[section]

\newcommand{\bbr}{\mathbb R}

\newcommand{\vast}{\bBigg@{4}}
\newcommand{\Vast}{\bBigg@{5}}

\begin{document}

\date{\today}

\subjclass[2020]{37H10, 37N40, 65K10} \keywords{Consensus-based optimization, global optimization, interacting particle systems, stochastic particle methods}

\thanks{\textbf{Acknowledgment.} {
J. Byeon was supported by the National Research Foundation of Korea grant funded by the Korean government (No.2019R1A6A1A10073437). S.-Y. Ha was partially supported by the National Research Foundation of Korea Grant (NRF-2020R1A2C3A01003881). J.-H. Won was supported by the Research Grant of Seoul National University.}
} 
\begin{abstract}
We propose Discrete Consensus-Based Optimization (DCBO), a fully discrete version of the Consensus-Based Optimization (CBO) framework. DCBO is a multi-agent method for the global optimization of possibly non-convex and non-differentiable functions. It aligns with the CBO paradigm, which promotes a consensus among agents towards a global optimum through simple stochastic dynamics amenable to rigorous mathematical analysis. Despite the promises, there has been a gap between the analysis of CBO and the actual behavior of the agents from its time-discrete implementation, as the former has focused on the system of continuous stochastic differential equations defining the model or its mean-field approximation. In particular, direct analysis of CBO-type algorithms with heterogeneous stochasticity is very challenging. DCBO distinguishes itself from these approaches in the sense that it has no continuous counterpart, thanks to the replacement of the ``softmin" operator with the ``hardmin" one, which is inherently discrete. Yet, it maintains the operational principles of CBO and allows for rich mathematical analysis. We present conditions, independent of the number of agents, for achieving a consensus or convergence and study the circumstances under which global optimization occurs. We test DCBO on a large number of benchmark functions to show its merits. We also demonstrate that DCBO is applicable to a diverse range of real-world problems, including neural network training, compressed sensing, and portfolio optimization, with competitive performance.
\end{abstract}

\maketitle \centerline{\date}

\section{Introduction}\label{sec:1}

We address the problem of global minimization of a real-valued objective function $f$, defined on a closed set $S \subset \mathbb{R}^d$:
\begin{equation}\label{A0}
	\min_{x \in S} f(x).
\end{equation}
We are interested in the case that $f$ is (Borel) measurable on $S$ and it can be evaluated at any point there. Importantly, we do not assume that $f$ is either convex or differentiable. We tackle problem \eqref{A0} by employing a large number of agents with random dynamics. Despite the random nature of the mechanism that includes intricate stochastic relations among the agents, we provide a rigorous analysis of the dynamics of the proposed method, including the proof of convergence to the global minimum. \newline
\indent Our method is motivated by the recent advances in \emph{consensus-based optimization} (CBO) \cite{CBO}, which belongs to a family of metaheuristic methods that prescribe instantaneous spatial dynamics among agents in terms of a system of stochastic differential equations (SDEs) to promote a consensus among the agents localized in the vicinity of a global minimizer. Compared to other metaheuristic methods \cite{DB,HM,KA,PSO}, the major advantage of CBO is its capacity enabling the construction of proofs of convergence to a global minimum in the absence of convexity or derivative information, largely due to a simpler mechanism. Notably, CBO methods are not merely derivative-free but also incorporate features that enhance their effectiveness. Totzeck et al. \cite{TW} utilize personal best information to enhance performance. Riedl \cite{R} incorporates memory effect and gradient information for refinement, and Carrillo et al. \cite{CJLZ} modifies the model to reduce its dependency on dimensionality, thereby adapting the CBO to applications in high-dimensional optimization problems, such as those encountered in machine learning tasks \cite{FHPS2}. Further, CBO variations on manifolds are also discussed in \cite{FHPS,FHPS3,HKKKY}.

\indent In general, CBO on $\bbr^d$ is represented as a system of SDEs of the following form:
\begin{align}\label{CBO}
    dX^i_t = \lambda F(\bar{X}_t - X^i_t) dt + \sigma G(\bar{X}_t - X^i_t) dW^i_t, \quad t \geq 0, \quad i \in \{1,\cdots,N\}=:[N],
\end{align}
where $X^i_t$ is the position of the $i$-th agent (particle) at time $t$, and $W^i_t$ is a $d$-dimensional Wiener process at time $t$ assigned to the $i$-th agent. Functions $F$ and $G$ are designed to be a zero vector and a zero matrix at the origin, respectively. Therefore, $\bar{X}_t = X^i_t$ implies $dX^i_t = 0$, so that each $X^i$ becomes stationary and terminates the updates when every agent arrives at a common point; this phenomenon is referred to as a \emph{consensus}. When the agents are close to the consensus state at time $T$, $\bar{X}_T$ is regarded as a candidate for a global minimizer. Motivated by the Laplace principle \cite{DZ}, the typical choice of consensus point $\bar{X}$ is by the action of the ``softmin" operator
\[
    \bar{X}_t = \frac{\sum_{k=1}^N X^k_t w_f^\beta(X^k_t)}{ \sum_{k=1}^N w_f^\beta(X^k_t)}, \quad w_f^\beta(X) := e^{-\beta f(X)}, \quad \beta > 0,
\]
which is a weighted average of particle's position process $X_t^1,\ldots,X_t^N$, where those that result in the smallest value of $f$ get the largest weight; as $\beta \to \infty$, $\bar{X}_t$ tends to the centroid of the particles that yields the smallest value of $f$.

\indent Due to the randomness induced by Wiener process, analysis of CBO is typically conducted by the corresponding mean-field approximation by limit $N \to \infty$. In this case, the SDE system \eqref{CBO} is approximated by the partial differential equation (PDE) corresponding to the McKean process, and the randomness embedded in SDE \eqref{CBO} converges to a deterministic diffusion form \cite{CCTT,FKR,GHV}. However, how far the PDE model can explain the actual operation of the implemented CBO (i.e., discretized SDE) remains theoretically subtle (see Section \ref{sec:2}). In this context, the actual operation of CBO has been studied indirectly at the PDE level and remains largely unsolved.

\indent The goal of this paper is to bridge this gap by introducing the \emph{Discrete CBO} (DCBO) model. Let $x^i_n$ be the (position of the) $i$-th agent at discretized time $n$. DCBO is described by following recurrence relations with randomness:
\begin{subequations}\label{update}
    \begin{align}
        &x^i_{n+1} = x^i_n + \gamma^1(p_n - x^i_n) + \gamma^2(p_n - x^i_n) \odot \eta_n^i, \label{DCBOa}\\
        \text{or} \quad &x^i_{n+1} = x^i_n + \bar{\gamma}^1(p_n - x^i_n) + \bar{\gamma}^2\|p_n - x^i_n\|\frac{\eta_n^i}{\sqrt{d}}, \quad n \in \mathbb{N} \cup \{0\}. \label{DCBOb}
    \end{align}
\end{subequations}
The third terms on the R.H.S. of \eqref{DCBOa} and \eqref{DCBOb} are referred to as anisotropic \cite{CJLZ} and isotropic diffusion \cite{CCTT,FKR,CBO}, respectively. Model \eqref{update} considers a mixture of two diffusions. More precisely, one of the two update rules in \eqref{update} is assigned to all agents a priori. In particular, the particle dynamics can be either fully anisotropic or isotropic, meaning that every agent can update via either $\eqref{DCBOa}$ or $\eqref{DCBOb}$. In \eqref{update}, $d$ represents the dimension of the ambient space, $\odot$ denotes the Hadamard product (i.e., entry-wise product), $\| \cdot \|$ denotes the Euclidean norm on $\mathbb{R}^d$, and $p_n$ is one of the $x^i_{n}$'s with the smallest function value. If there are multiple agents with the smallest function value, we take the one with the smallest index. Each $\eta^i_n$ is a $d$-dimensional vector of i.i.d. standard normal random variables. The coefficients $\gamma^1 > 0$ and $\gamma^2 \geq 0$ control the deterministic and stochastic behaviors of the agents, respectively. The core idea of DCBO is to replace the softmin operator with the ``hardmin operator'' by letting $\beta \to \infty$:
\begin{align}\label{lim}
	\bar{x}_n = \frac{\sum_{k=1}^N x^k_n e^{-\beta f(x^k_n)}}{ \sum_{k=1}^N e^{-\beta f(x^k_n)}} 
	~ \xrightarrow{\beta \to \infty} ~
	\argmin_{i \in \{1,\ldots,N\}}f(x_n^i) = p_n.
\end{align}
Thus, relations \eqref{update} can be informally understood as modeling the CBO under $\beta = \infty$, anticipating enhanced performance since larger $\beta$ values are known to improve CBO-type models \cite{CJLZ, CBO, TW}. However, mathematical analysis in this scenario has been studied in a very limited scope, due to the discontinuity of $(x^1_t,\ldots,x^N_t) \mapsto p_t$ for a continuous-time $t$. Also note that \eqref{lim} holds true when all $f(x^i_n)$ values are distinct. As a result, hardmin-based models \emph{cannot} be derived from the time-discretization of continuous models where mathematical analysis is feasible (see Section \ref{sec:2} for details). Thus, DCBO is inherently discrete and not a \emph{discretized} version of CBO. Despite this distinction, DCBO phenomenologically mirrors CBO as it shares the form \eqref{CBO} and is approximated from the CBO model by \eqref{lim}.\newline
\indent  CBO-type algorithms typically terminate on an approximate consensus, i.e., $\|x^i_n-x^j_n\|<\varepsilon \ll 1$ for any $i$ and $j$. A consensus is a prerequisite for the convergence of the stochastic process. The precise definition is follows. 

\begin{definition}[Consensus and convergence in DCBO]\label{D3.1}
	Let $(x_n^i)_{i \in [N], ~ n \geq 0}$ be a solution to the discrete stochastic process \eqref{update} with an underlying probability space $(\Omega,\mathcal{F},\mathbb{P})$. Let $X_n(\omega) = (x_n^1(\omega),\ldots,x_n^N(\omega))$ and $\omega \in \Omega$ be a sample point.
	\begin{enumerate}
		\item We say that a sample path $(X_n(\omega))$ reaches a \emph{consensus} if
		\[
			\lim_{n \to \infty}\|x_n^i(\omega)-x_n^j(\omega)\| = 0, \quad \forall i,j \in [N].
		\]
		\item We say that a sample path $(X_n(\omega))$ \emph{converges} if there exists a vector $x_\infty(\omega)$ such that
		\[
			\lim_{n \to \infty}\|x_n^i(\omega)-x_\infty(\omega)\| = 0, \quad \forall i \in [N].
		\]
	\end{enumerate}
\end{definition}

In the study of CBO-type models, the main theoretical interest is to propose suffiicient conditions under which consensus would occur and, when it does, if it would occur at an optimum. However, a rigorous analysis remains largely unsolved for particle-level models. Our key contribution lies in providing a CBO-type particle model with which almost every sample path exhibits a consensus under some reasonable assumptions. The main results regarding \eqref{update} are summarized as follows.
\begin{theorem}[Informal summary]\label{T1.1}
	Let $f:\bbr^d \to \bbr \cup \{+\infty\}$ be a Borel measurable function with a closed effective domain $S$. Suppose that $\rho_{\mathrm{in}}$ is an initial probability distribution on the search space $\mathbb{R}^d$ such that each initial position $x^i_0$ is independently drawn from $\rho_{\mathrm{in}}$, and $\mathrm{supp}(\rho_{\mathrm{in}}) \subset S$. Let $\omega \in \Omega$ and $(x^i_n(\omega))$ be a sample point and a sequence of iterates generated by \eqref{update}, respectively. Then, a suitable set of parameters $\Gamma$ and its subset $\Gamma'$ exist, independent of the number of agents $N$ and dimension $d$, for which the following assertions hold:
\begin{enumerate}
    	\item For $(\gamma^1,\bar{\gamma}^1,\gamma^2,\bar{\gamma}^2) \in \Gamma$, almost every $\omega$ ensures that if the iterate cannot improve the best value $f(p_n(\omega))$, the sample path converges.
	\item Let $f^*$ be the essential infimum of $f$ with respect to the Lebesgue measure. Suppose that $(\gamma^1, \bar{\gamma}^1, \gamma^2, \bar{\gamma}^2) \in \Gamma'$ and some initial position $x^i_0$ is contained in a bounded sub-level set of $f$. Then for almost every $\omega$, either the sample paths reach a consensus or $f(p_n)$ converges to $f^*$ or below, provided that the noises are fully isotropic. In particular, if $f$ is continuous on $S$ with a unique global minimizer $x^*$ and $f^* = f(x^*)$ holds, then a consensus emerges almost surely.
    \item Suppose that $f$ is continuous on $S$ and a global minimizer exists in the support of $\rho_{\mathrm{in}}$. Then $f_\infty := \lim_{n \to \infty}f(p_n)$ exists and is finite, and for any $\varepsilon > 0$, the probability that
    	\begin{align}\label{X}
    		f_\infty < \min f + \varepsilon
    	\end{align}
    is not less than $1-(1-P_\varepsilon)^N$, where $P_\varepsilon > 0$ is a positive constant that may depend on $\rho_{\mathrm{in}}$, dimension $d$, regularity of $f$, and $\varepsilon$.
    \item Suppose that a continuous function $f_m$ exists such that $f_m(0) = 0$, $f_m(x - x^*) \leq f(x) - f(x^*)$ for any global minimizer $x^*$ of $f$, and $f_m$ has a strict local minimum at the origin. If $f$ has a non-empty bounded sub-level set, then the distance between $p_n$ and a set of global minimizers converge to 0 as $\varepsilon \searrow 0$, provided that the relation \eqref{X} occurs.
    \end{enumerate}
\end{theorem}

The precise definitions of $\Gamma$ and $\Gamma'$ are provided in Section \ref{sec:3}. The first two statements of Theorem \ref{T1.1} guarantee a consensus in practice, while the last two statements present a lower bound for the performance of \eqref{update}. Since $f(p_n)$ is precisely the smallest value among $f(x^i_m)$ for $ i \in [N]$ and $m \in \{0,1,\ldots,n\}$, $f(p_n)$ is monotonically decreasing in $n$ (see Proposition \ref{P3.2}) and one may assume a priori that updates will eventually cease. Under this assumption, the first statement indicates that a consensus occurs when the parameters belong to $\Gamma$. Furthermore, it can be inferred from the proof that the algorithm terminates soon when some agents locate the global minimizer early (see Remark \ref{R3.3} and Remark \ref{R4.8}). The second statement is distinct from the first statement in that it applies irrespectively of whether or not $f(p_n)$ improves. When CBO \eqref{CBO} is approximated by a PDE, which is informally understood as a CBO with infinitely many agents, the latter usually guarantees the localization of the diffusion process at a global minimizer, if it is contained in the support of the initial distribution. The third and fourth statements can be understood as the particle-level counterparts of this phenomenon for a finite $N$. \newline
\indent The contributions of this paper are as follows: first, we provide sufficient conditions for consensus or convergence that do not depend on the number of particles $N$. Previous approaches have been limited to rather restrictive conditions: homogeneous noises \cite{BHKLMY, HJK,HJK2}, or a diminishing impact of noise over time \cite{CJL,HHK}, or increasing $N$ \cite{KHJK}. Our results are not bound by these limitations. Second, we demonstrate the efficacy of our model even when multiple dynamics \eqref{DCBOa} and \eqref{DCBOb} are employed simultaneously, enhancing exploration capabilities. Third, we establish a quantitative lower bound for the probability of achieving global optimization. None of these aspects have been previously investigated at the particle level.

The rest of the paper is organized as follows. Section \ref{sec:2} reviews previous models related to \eqref{update} and introduces the main algorithm. Section \ref{sec:3} studies conditions for the emergence of a stochastic consensus state. Section \ref{sec:4} presents conditions under which DCBO achieves global optimization. Section \ref{sec:5} provides numerical simulations and applications to portfolio optimization, neural networks and compressed sensing. Finally, Section \ref{sec:6} concludes the paper with a discussion. \newline

\textbf{Notation and assumptions.} Throughout the paper, $\|\cdot\|$ denotes the standard $\ell^2$-norm on $\mathbb{R}^d$. For $m \in \mathbb{N}$, we denote $[m] := \{1, 2, \ldots, m\}$. The dimension of the search space and the number of particles are denoted by $d$ and $N$, respectively.
We assume $f$ to be the objective function to be optimized, which is at least Borel measurable. We identify a real-valued function $f$ on $S$ with $f: \mathbb{R}^d \to \mathbb{R} \cup \{+\infty\}$, where the effective domain of $f$ is $S$. The essential infimum of $f$ with respect to the Lebesgue measure is denoted by $f^*$, and we only consider functions for which $f^* < +\infty$. If $f^* \in \mathbb{R}$, the set $S_\varepsilon^f$ denotes the sub-level set $\{x \mid f(x) \leq f^* + \varepsilon\}$. The initial probability distribution (measure) on the search space $\mathbb{R}^d$ is denoted as $\rho_{\mathrm{in}}$, and each initial position $x^i_0$ is drawn from $\rho_{\mathrm{in}}$ for $i \in [N]$, where $\mathrm{supp}(\rho_{\mathrm{in}}) \subset S$. Multivariate normal distribution of dimension $d$ with mean vector $\mu$ and covariance matrix $\Sigma$ is denoted by $\mathcal{N}_d(\mu,\Sigma)$. We denote the underlying probability space of DCBO by $(\Omega, \mathcal{F}, \mathbb{P})$, where each sample point in $\Omega$ corresponds to a product of $N$ independent copies of the initial distribution $\rho_{\mathrm{in}}$ and a sequence of $\eta^i_n \overset{\mathrm{i.i.d.}}{\sim} \mathcal{N}_d(0,I)$. We denote a sample point and the corresponding sequence of iterates generated by \eqref{update} by $\omega$ and $(x_i^n(\omega))$, respectively.

\section{Preliminaries} \label{sec:2}
\setcounter{equation}{0}

In this section, we review state of the art in the CBO models and introduce the DCBO algorithm.

\subsection{Previous results} The design of CBO type algorithms involves substituting appropriate functions into $F$ and $G$ in \eqref{CBO}, and, if necessary, adding additional terms to incorporate further effects. The analysis of CBO algorithms is considered at three levels:

\begin{enumerate}
    \item Continuous-time SDE: The seminal work \cite{CBO} introduces CBO as a system of continuous SDEs. For suitable functions $F:\mathbb{R}^d \rightarrow \mathbb{R}^d$ and $G:\mathbb{R}^d \rightarrow \mathbb{R}^{d \times d}$, it reads as:
    \begin{align}\label{SDECBO}
        dx^i_t = \lambda F(\bar{x}_t - x^i_t) dt + \sigma G(\bar{x}_t - x^i_t) dW^i_t, \quad
        \bar{x}_t = \frac{\sum_{k=1}^N x^k_t e^{-\beta f(x^k_t)}}{\sum_{k=1}^N e^{-\beta f(x^k_t)}}.
    \end{align}
    \item PDE: When the number of agents $N$ is sufficiently large, CBO can be approximated by a single PDE via the stochastic mean-field limit. When $G(x)$ is a diagonal matrix, the empirical mean of solutions to \eqref{SDECBO} can be effectively approximated by $\rho$ which is a solution of the following PDE:
    \begin{align}\label{PDECBO}
        \partial_t \rho_t = -\lambda \nabla \cdot (F(\bar{x}[\rho_t]-x) \rho_t) + \frac{\sigma^2}{2}\sum_{k=1}^d \partial_{kk} \big((G(\bar{x}[\rho_t]-x)_{kk}^2 \rho_t\big)
    \end{align} ~ \vspace{-.7cm}
    \begin{align*}
        \bar{x}[\rho_t] = \frac{\int_{\mathbb{R}^d} x e^{-\beta f(x)} \, d\rho_t(x)}{\int_{\mathbb{R}^d} e^{-\beta f(x)} \, d\rho_t(x)}, \quad
        M_{kk} := \text{the $k$-th diagonal entry of matrix $M$.}
    \end{align*}
    \item Discretized SDE: Since the actual implementation of the continuous-time CBO algorithm is discrete, this form is especially relevant for direct optimization applications. They are usually derived by applying the Euler-Maruyama discretization to \eqref{SDECBO}:
    \begin{align}\label{discCBO}
        x^i_{n+1} = x^i_n + h F(\bar{x}_n - x^i_n) + \sqrt{h} \sigma G(\bar{x}_n - x^i_n) \eta^i_n, \quad
        \eta_n^i \overset{\mathrm{i.i.d.}}{\sim} \mathcal{N}_d(0,I).
    \end{align}
\end{enumerate}
	In the above models, typically $F(x) = x$ and $G$ is selected from one of the following two forms:
\[
    G(x) =
    \begin{cases}
        \|x\|I_d, & \text{for isotropic diffusion}, \\
        \mathrm{diag}(x), & \text{for anisotropic diffusion},
    \end{cases}
\]
where $I_d \in \mathbb{R}^{d \times d}$ is the identity matrix, and $\mathrm{diag}(x)$ is the diagonal matrix with the vector $x$ as its diagonal. These two types of diffusion offer comparative advantages depending on the situation: isotropic diffusion allows for more exploratory search across the domain, while anisotropic diffusion can efficiently search high-dimensional spaces \cite{CJLZ}.

The analysis of \eqref{SDECBO} and \eqref{discCBO} poses huge challenges; the consensus and convergence analysis largely remains an open question. Consequently, much of the research has pivoted towards the study of the approximated PDE \eqref{PDECBO}, which is easier to analyze as the inherent randomness converges to a deterministic diffusion form. Notably, the works \cite{CCTT, FKR, FKR22} demonstrate that $\rho_t$ tends to concentrate near the global minimizer when the initial distribution includes a global minimizer in its support. Such mean-field analysis provides a substantial insight into the operation of CBO. 

Nevertheless, it is challenging to determine how far the results from the mean-field model apply to the actual workhorse \eqref{discCBO}. Primarily, the assumption on the macroscopic equation \eqref{PDECBO} that the initial distribution $\rho_{\mathrm{in}}$ includes a global minimizer inherently presupposes knowledge of its location. Moreover, the clarity on to what extent \eqref{discCBO} can perform global optimization with a certain probability based on the analysis of \eqref{PDECBO} also remains uncertain. To the authors' knowledge, the following Proposition is the sole result that quantifies the probability of global optimization in \eqref{discCBO} from the analysis of \eqref{PDECBO}.

\begin{proposition}[\cite{FKR}]\label{P2.1.5}
    Let $ ((X_{k\Delta t}^i)_{k=0,\ldots,K})_{i \in [N]} $ be the iterates generated by \eqref{discCBO} with isotropic diffusion and time-step $\Delta t$, where $K$ matches the time horizon $K\Delta t = T$. Assume that the initial data and system parameters are well-prepared, and $f$ satisfies suitable regularity conditions. Then, we have
    \[
        \left\| \frac{1}{N}\sum_{i=1}^N X^i_{K \Delta t} - x^* \right\|^2 \leq \varepsilon_{\mathrm{total}},
    \]
    with probability at least $ 1-(\delta + \varepsilon^{-1}_{\mathrm{total}}(C_{\mathrm{NA}}\Delta t + C_{\mathrm{MFA}}N^{-1}+6\varepsilon)) $, where $\delta, \varepsilon, \varepsilon_{\mathrm{total}}, C_{\mathrm{NA}}$, and $C_{\mathrm{MFA}}$ are positive constants, and $C_{\mathrm{NA}}$ depends linearly on the dimension $d$ and the number of agents $N$.
\end{proposition}

\indent The tendency for $\sum_{i=1}^N X^i_{K \Delta t}/N$ to converge to a global minimizer is proved by the analysis of PDE \eqref{PDECBO}. The terms $C_{\mathrm{NA}}\Delta t$ and $C_{\mathrm{MFA}}N^{-1}$ represent errors arising from time-discretization and mean-field approximation, respectively. Proposition \ref{P2.1.5} presents a dilemma. Consider a scenario where all parameters except for $N$ are constant. As $C_{\mathrm{NA}}$ depends linearly on $N$, the probability bound for the success of global optimization decreases to negative as $N$ increases. This contradicts the intuition that a larger $N$ should yield better performance. This lower bound reflects the fact that a discretization error is proportional to the dimension of the state space (in this case, $dN$), rather than an actual performance degradation with an increasing $N$. In summary, the current state of the art reveals a gap in explaining implemented CBO \eqref{discCBO} through PDE \eqref{PDECBO}.

	Therefore, analyzing CBO at the discrete level is crucial for gaining insights into its actual behavior of the algorithm. However, due to the mathematical complexities involved, discrete models have been studied only partially. Ko et al. \cite{KHJK} demonstrated that a consensus occurs almost surely if the second moment of noise diminishes sufficiently as $N$ increases, and Bae et al. \cite{BHKLMY} explored the likelihood of a consensus and global optimization when noise is homogeneous (i.e., $\eta^i_n = \eta^j_n$ for all $i \neq j$). However, these conditions inherently limit the exploration effect, ensuring consensus under significantly constrained circumstances: the work \cite{KHJK} considers a scenario where noise becomes infinitesimally small as $N$ approaches infinity, and the authors in \cite{BHKLMY} simplifies the analysis at the cost of performance degradation caused by the significant under-exploration of the search space (we refer to Section \ref{sec:5} for empirical results). 

	To summarize the discussion so far, many of the theoretical properties of the actual CBO algorithm remains uncharted. This motivates us to introduce the DCBO model in which mathematical analysis is feasible.

%

\subsection{The model}\label{sec:2.2} Our model is defined via the following two diffusion maps 
\begin{align}\label{model}
\begin{cases}
	&F_1(x,y,\eta) = x + \gamma^1(y - x) + \gamma^2(y - x) \odot \eta,\\
	&F_2(x,y,\eta) = x + \bar{\gamma}^1(y - x) + \bar{\gamma}^2\|y - x\|\frac{\eta}{\sqrt{d}}.
\end{cases}
\end{align}

Note that $F_1$ can be formally derived from \eqref{discCBO} by taking the limit \eqref{lim} and substituting $h\lambda$ and $\sqrt{h}\sigma$ for $\gamma^1$ and $\gamma^2$, respectively. The derivation of $F_2$ follows a similar procedure. Note that $F_1$ and $F_2$ correspond to anisotropic and isotropic diffusions, respectively. For a nonempty closed set $S$, we define the DCBO recurrence as follows.
\begin{align}\label{DCBO}\tag{DCBO}\begin{aligned}
	&\displaystyle p_n := x^{\min \mathcal{I}_n}_n, \quad \mathcal{I}_n := \{ i \in [N] \mid f(x^i_n) = \min_{j \in [N]} f(x^j_n) \}, \vspace{.2cm} \\
	&x^i_{n+1} = F^i(x^i_n,p_n,\eta^i_n), \quad F^i = F_1 ~ \text{or} ~ F_2,
\end{aligned}\end{align}
where $x^i_0 \in S, \quad \eta_n^i \overset{\mathrm{i.i.d.}}{\sim} \mathcal{N}_d(0,I), \quad n \in \mathbb{N} \cup \{0\}, \quad i \in [N]$. In other words, we define the consensus point $p_n$ as the location of the agent with a smallest function value, and with the smallest index. Without loss of generality, we assume that $f$ takes the value of $+\infty$ outside $S$. If $S$ is a convex set so that the projection $\mathcal{P}_S$ onto $S$ is defined, then we further update $x^i_{n+1}$ by the projection for each iteration. That is, we change the second line of \eqref{DCBO} into
\begin{align}\label{proj}
	x^i_{n+1} = \mathcal{P}_S(F^i(x^i_n,p_n,\eta^i_n)).
\end{align}
Model \eqref{DCBO} employs a mixture of two discrete diffusion processes. The system parameters $\gamma^1,\gamma^2,\bar{\gamma}^1,\bar{\gamma}^2$ are positive constants representing drift and noise intensity, respectively. All the noises are external: $\eta^i_n$ is independent of $x^k_m$ and $\eta^{k'}_{m'}$ for all $k,k' \in [N], k' \neq i$ and $m \in [n] \cup \{0\}, ~ m' \in \mathbb{N} \cup \{0\}$. The pseudocode for the model \eqref{model} is presented in Algorithm \ref{Alg1}.

\begin{algorithm}
\caption{DCBO}\label{Alg1}
\begin{algorithmic}[1]
    \State \textbf{Input:} function $f$, closed domain $S$, parameters $(\gamma^1,\gamma^2,\bar{\gamma}^1, \bar{\gamma}^2)$,  the number of agents $N$, stopping criterions $\text{max\_iter}$ and $\text{max\_dist}$
    \State Initialize $x^i$ in $S$ for all $i \in [N]$
    \State $p \leftarrow x^{\min \mathcal{I}}, \quad \text{where} \quad \mathcal{I} = \{ i \in [N] \mid f(x^i) = \min_{j \in [N]} f(x^j) \}$
    \State $n \leftarrow 0$
    \While{$n < \text{max\_iter}$ and $\max \|x^i-p\| \geq \text{max\_dist}$}
        \State $x^i \leftarrow F^i(x^i,p,\eta^i_n)$ for all $i \in [N]$
        \If{Projection $\mathcal{P}_S$ into $S$ is defined}
            \State $x^i \leftarrow \mathcal{P}_S(x^i)$ for all $i \in [N]$
        \EndIf
        \State Update $p \leftarrow x^{\min \mathcal{I}}, \quad \text{where} \quad \mathcal{I} = \{ i \in [N] \mid f(x^i) = \min_{j \in [N]} f(x^j) \}$
        \State $n \leftarrow n+1$
    \EndWhile
    \State \textbf{return} $(p,f(p))$.
\end{algorithmic}
\end{algorithm}

It is worth re-emphasizing that \eqref{DCBO} is \emph{not} a discretization of a SDE and inherently discrete. To see this, suppose that \eqref{DCBO} can be derived from some continuous SDE and consider $p_t$, a continuous counterpart of $p_n$. For continuous SDE models, it is not evident whether $p_t = X^{\min \mathcal{I}_t}_t$, where $\mathcal{I}_t := \{ i \in [N] \mid f(X^i_t) = \min_{j \in [N]} f(X^j_t) \}, ~ t \geq 0$, can be utilized as a consensus point since the map $t \mapsto p_t$ may be discontinuous even if $f$ is continuous; usually a local Lipschitz continuity of the coefficient functions is required for a solution to \eqref{SDECBO} to be well-defined. Even worse, at the mean-field level, it is difficult to define either $\mathcal{I}_t$ or $p_t$ on a continuum. Therefore, \eqref{DCBO} cannot be defined as a discretization of the PDE model \eqref{PDECBO} either.

	Algorithm \ref{Alg1} has several advantages compared to the conventional CBO algorithms:
\begin{enumerate}
    \item Expected performance: analyses of conventional CBO suggest better performance for larger values of $\beta$ \cite{CJLZ, CBO, TW}. Consequently, enhanced performance can be expected under the regime $\beta = \infty$, which corresponds, at least phenomenologically, to our model.
    \item Analysis: construction of $p_n$ facilitates mathematical analysis of \eqref{discCBO} previously deemed unfeasible by allowing for consistently choosing one element from the set of minimizers, ensuring agents to exhibit convergence or consensus (see Proposition \ref{P3.2} and Theorem \ref{T3.3}).
    \item Simplified assumptions: conditions on DCBO are straightforward compared to continuous SDE or PDE models \eqref{SDECBO} and \eqref{PDECBO}. Given the requirement of local Lipschitz continuity for the coefficient functions, $f$ must be at least locally Lipschitz to ensure the existence of a solution in continuous models. In \eqref{PDECBO}, the conditions for parameters $\rho_{\mathrm{in}},\beta,\lambda$,$\sigma$ and the regularity condition of $f$ achieving global optimization are complexly interwound, making it challenging to identify a set of parameters that satisfy these conditions. In contrast, DCBO permits consensus/convergence under clear and intuitive conditions when the function is merely continuous or measurable.
\end{enumerate}

\begin{remark}
    DCBO is designed to terminate at the $n$-th iteration when the agents are close to a consensus state, i.e., $\max_{i,j}\|x^i_n - x^j_n\| \leq \varepsilon$ for a small $\varepsilon$. To validate the consensus, it is sufficient to observe $\max_{i}\|x^i_n - p_n\| \leq \varepsilon/2$ by applying the triangle inequality. This latter condition is employed in Algorithm \ref{Alg1} as it incurs lower computational cost. Notably, in light of the consensus criterion (see Proposition \ref{P3.4}), the emergence of a consensus might be validated by
    \[
        \|p_{n+1} - p_n\| < \varepsilon \text{ for sufficiently many consecutive instances of } n,
    \]
which further reduces computational cost.
\end{remark}

\section{Emergence of a consensus}\label{sec:3}
\setcounter{equation}{0}

In this section, we study several conditions for Algorithm \ref{Alg1} to achieve a consensus. First, we prove that a sample path exhibits a convergnece under some mild assumptions. Then we propose a criterion in which a consensus would occur.

\subsection{Convergence analysis}

As discussed in Section \ref{sec:2.2}, selecting a ``current best point'' $p_n$ as the consensus point facilitates mathematical analysis. This is based on the fact that $p_n$ cannot escape a sub-level set, which is illustrated by the following simple proposition.

\begin{proposition}\label{P3.2} Let $(x^i_n)$ be a sequence of iterates generated by Algorithm \ref{Alg1}. Then $f(p_n)$ is monotone decreasing in $n$. In particular, $p_n \in S$ for each $n \in \mathbb{N}$.
\end{proposition}

\begin{proof}
	Let $i_n$ be a random index satisfying $p_n = x^{i_n}_n$. Then for each event, $x^{i_n}_{n+1} = F^{i_n}(x_n^{i_n}, p_n, \eta_n^{i_n}) = x^{i_n}_n$ and
	\[
		f(p_n) = f(x^{i_n}_n) = f(x^{i_n}_{n+1}) \geq f(p_{n+1}). 
	\]
	Since $f(p_0) < \infty$, we have $f(p_n) < \infty$ for each $n$, and therefore $p_n \in S$.
\end{proof}

	Proposition \ref{P3.2} states that $f(p_n)$ represents the best objective value not only at time $n$ but also for all preceding times up to $n$. Hereafter, we indicate that Algorithm \ref{Alg1} updates the optimal value (at time $n$) if $f(p_{n+1}) < f(p_n)$. \newline

Intuitively, if the deterministic effect $\gamma^1$ or $\bar{\gamma}^1$, which drags each agent to $p_n$, is too small compared to the intensity of external noise $\gamma^2$ or $\bar{\gamma}^2$, then one cannot expect a consensus. In order to provide appropriate parameter candidates, we will proceed under the following conditions: \vspace{.5cm}

\textbf{Condition.} The system parameters $(\gamma^1,\gamma^2)$ and $(\bar{\gamma}^1,\bar{\gamma}^2)$ satisfy the following conditions.
\begin{subequations}\label{B1}
    \begin{align}
        &\gamma^1 \in (0,1), \quad \gamma^2 \geq 0, \quad \inf_{\alpha > 0}\mathbb{E}[|1- \gamma^1 + \gamma^2 \eta|^\alpha] < 1 \label{B1a}\\
       &\bar{\gamma}^1 \in (0,1), \quad \bar{\gamma}^2 \geq0, \quad \bar{\gamma}^1 \geq \bar{\gamma}^2.  \label{B1b}
    \end{align}
\end{subequations}
where $\eta \sim \mathcal{N}_1(0,1)$. For example, $\eqref{B1a}$ is satisfied if
\begin{align}
	&\mathbb{E}[|1- \gamma^1 + \gamma^2 \eta|]
	= \gamma^2\sqrt{\frac{2}{\pi}}\exp\left( -\frac{(1-\gamma^1)^2}{2(\gamma^2)^2}\right) + \gamma^1\left( 1 - 2\Psi\left( -\frac{\gamma^1}{\gamma^2}\right) \right) < 1, \label{B2}\\
	&\mathbb{E}[|1- \gamma^1 + \gamma^2 \eta|^2] = (1-\gamma^1)^2 + \gamma^2 < 1, \label{B3}
\end{align}
where $\Psi$ is the normal cumulative distribution, i.e.
$
	\Psi(x) := \frac{1}{\sqrt{2\pi}}\int_{-\infty}^x e^{-\frac{t^2}{2}}dt.
$ ~ \newline

\indent Proposition \ref{P3.2} states that $f(p_n)$ decreases its value whenever it updates. Therefore, for any proper function $f$, it is natural to assume a priori that $f(p_n)$ would terminate its update eventually. In this situation, the following theorem asserts that consensus is attained under very mild conditions.

\begin{theorem}\label{T3.3}
	Suppose that $f$ is a Borel measurable function, and let $(x^i_n)$ be the iterates generated by Algorithm~\ref{Alg1}. Then, for almost every $\omega \in \Omega$, either the current best objective vlaue $f(p_n)$ keeps updating, or $(x^i_n(\omega))$ converge.
\end{theorem}

\begin{proof}
	It suffices to prove that at least one of the following event occurs almost surely:
	\begin{enumerate}
		\item[(a)] $(f(p_n))$ is updated infinitely often. That is, there exists a strictly decreasing subsequence of $f(p_n)$ such that
		\[
			f(p_{n_m}) > f(p_{n_{m+1}}), \quad n_m < n_{m+1}, \quad \forall m \in \mathbb{N}, \quad \lim_{m \to \infty}n_m = \infty. 
		\]
		\item[(b)] There exists a vector $x_\infty \in S$ such that $\lim_{n \to \infty}\max_{i}\|x^i_n-x_\infty\| = 0$. \vspace{.2cm}
	\end{enumerate}
	
\indent We first consider an agent $x_n^i$ satisfying $F^i=F_1$. Suppose that projection $\mathcal{P}_S$ is not defined. Since $f$ is Borel measurable, $\{ \omega \in \Omega \mid f(p_n(\omega)) > f(p_{n+1}(\omega)) \} \in \mathcal{F}$ for each $n \in \mathbb{N}$ so that $\Omega_a:=\{ \omega \in \Omega \mid f(p_n(\omega)) > f(p_{n+1}(\omega)) ~ \text{infinitely often} \} \in \mathcal{F}$ as well. Therefore the probability of event (a) is well defined. If $\mathbb{P}[\Omega_a]=1$, then there is nothing to prove. Assume $\mathbb{P}[\Omega_a]<1$ and fix $\omega \notin \Omega_a$. We claim that for some fixed vector $p_\infty = p_\infty(\omega)$,
	\begin{align}\label{convclaim}
		\text{$p_n=p_\infty$ for all but finitely many $n \in \mathbb{N}$.}
	\end{align}
	($\diamond$ Proof of \eqref{convclaim}) To verify the claim \eqref{convclaim}, it suffices to observe that for some $N_\omega \in \mathbb{N}$, there exists a fixed index $i_\infty = i_\infty(\omega) \in [N]$ such that
		\begin{align}\label{subclaim}
			f(p_n(\omega)) = f(x^{i_\infty}_n(\omega)), \quad f(p_n(\omega)) \notin \{ f(x_n^i) \mid i < i_\infty \}, \quad \forall n > N_\omega.
		\end{align}
	To see that \eqref{subclaim} implies \eqref{convclaim}, let $i_n$ be a random index satisfying $p_n = x^{i_n}_n$. Then \eqref{subclaim} and the construction of \eqref{DCBO} states that $i_n = i_\infty$ for $n > N_\omega$. This leads to
	\begin{align*}
		p_{n+1}=x^{i_{n+1}}_{n+1} =x_{n+1}^{i_\infty} &= F^{i_\infty}(x_n^{i_\infty}, p_n, \eta_n^{i_n}) \\
		&= F^{i_\infty}(x_n^{i_\infty}, x_n^{i_\infty}, \eta_n^{i_n})  = x^{i_\infty}_n = x^{i_n}_n = p_n, \quad \forall n > N_\omega,
	\end{align*}
	which proves \eqref{convclaim}. Therefore we focus on the proof of \eqref{subclaim}. Since $f(p_n)$ is non-increasing in $n$ and $\omega \notin \Omega_a$, $f(p_n)$ converge in a finite time. Therefore for some $N_\omega' \in \mathbb{N}$, $f(x^{i_n}_n)=f(p_n)$ is a constant for $i_n \in \mathcal{I}_n$, $n > N_\omega'$, where $\mathcal{I}_n=\mathcal{I}_n(\omega)$ is defined in \eqref{DCBO}. From the construction of \eqref{DCBO} and definition of $\mathcal{I}_n$, it is straightforward to check
	\begin{align*}
	\min\mathcal{I}_n
		\begin{cases}
			> \min\mathcal{I}_{n+1}, \quad &\text{if $f(x^{j}_{n+1}) = f(p_{n+1})$ for some $j <  \min\mathcal{I}_n$}, \\
			= \min\mathcal{I}_{n+1}, \quad &\text{otherwise}.
		\end{cases}
		\quad \forall n > N_\omega'.
	\end{align*}
	Since $\mathcal{I}_n \subset [N]$ and $\min\mathcal{I}_{n} \geq \min\mathcal{I}_{n+1}$, the limit $i_{\infty} := \lim_{n \to \infty} \min\mathcal{I}_{n}$ exists and $i_{\infty} = \min\mathcal{I}_{n}$ for all but finitely many $n$. Thus for some $N_\omega (> N'_\omega)$, we have $p_n = x^{i_{\infty}}_n$ whenever $n > N_\omega$. Then the minimality of $i_\infty$ proves \eqref{subclaim} and the claim \eqref{convclaim}. \newline
	
Now from \eqref{convclaim}, one has
	\begin{align*}
		x^i_{n+1} - p_\infty
		= (1- \gamma^1)(x^i_{n} - p_\infty) - \gamma^2(x^i_{n} - p_\infty) \odot \eta^{i}_n, \quad \forall n > N_\omega.
	\end{align*}
For a vector $v^i_n \in \bbr^d$, let its $k$-th component be $v^{i,k}_n$. Then for each $k \in [d]$, we have
	\begin{align*}
		&x^{i,k}_{n+1} - p^k_\infty
		= (1- \gamma^1 - \gamma^2 \eta^{i,k}_n)(x^{i,k}_{n} - p^k_\infty), \quad \forall n > N'_\omega \\
		&\Longrightarrow \quad |x^{i,k}_{n+1} - p^k_\infty|
		\leq |1- \gamma^1 - \gamma^2 \eta^{i,k}_n||x^{i,k}_{n} - p^k_\infty|=:X^{i,k}_n|x^{i,k}_{n} - p^k_\infty|, \quad \forall n > N'_\omega \\
		&\Longrightarrow \quad |x^{i,k}_{m} - p^k_\infty|
		\leq X^{i,k}_{m-1}X^{i,k}_{m-2}\cdots X^{i,k}_n |x^{i,k}_{n} - p^k_\infty|, \quad \forall m > n > N_\omega.
	\end{align*}	
	We then show that for some $\beta > 0$, the product $(X^{i,k}_{m-1}X^{i,k}_{m-2}\cdots X^{i,k}_n)^\beta$ converges to zero as $m \to \infty$ almost surely. Recall from the assumption \eqref{B1} that $\inf_{\alpha > 0}\mathbb{E}[|1- \gamma^1 - \gamma^2 \eta|^\alpha] < 1$, which implies $\mathbb{E}[(X_n^{i,k})^\beta] < 1$ for some $\beta>0$. Using $\log{x} \leq x-1$ for $x > 0$, we have
\begin{align}\begin{aligned}\label{B3-1}
	\prod_{l = m-1}^{n} (X^{i,k}_l)^\beta
	\leq \exp\left((n-m+2)\sum_{l = m-1}^{n} \frac{(X^{i,k}_l)^\beta - 1}{n-m+2} \right)
	\xrightarrow[~]{\mathrm{a.s.}} 0,
\end{aligned}\end{align}
where the almost sure limit can be obtained from $\mathbb{E}[(X^{i,k}_n)^\beta - 1] < 0$ and the strong law of large numbers. Therefore $|x^{i,k}_{n} - p^k_\infty|^\beta$ converge to 0 almost surely, conditioned on the event $\Omega \backslash \Omega_a$. We apply this argument to each $k$ and pose $x_\infty = p_\infty$ to obtain the desired result. When $S$ is a convex set, \eqref{proj} kicks in. The proof is similar since the projection map is non-expansive for each coordinate: $| \mathcal{P}_S({x}^{i,k}_n) - \mathcal{P}_S({x}^{j,k}_n)|  \leq |{x}^{i,k}_n - {x}^{j,k}_n|$.

	\indent Now consider the agent $x_n^j$ such that $F^j=F_2$. Since the basic strategy is similar to the previous case, we only outline the proof. Fix $\omega \notin \Omega_a$. Then for all but finitely many $n$, we have
	\begin{align*}
		x^{j,k}_{n+1} - p^k_\infty
		= (1- \bar{\gamma}^1)(x^{j,k}_{n} - p^k_\infty) + \bar{\gamma}^2\|p_\infty-x_n^{j}\|{\eta^{j,k}_n}/{\sqrt{d}}.
	\end{align*}
	Square both sides above and sum over $k \in [d]$. The Cauchy-Schwarz inequality entails
	\begin{align}\begin{aligned}\label{B4}
		\|x^{j}_{n+1} - p_\infty\|^2
		\leq& (1- \bar{\gamma}^1)^2\|x^{j}_{n} - p_\infty\|^2 + (\bar{\gamma}^2)^2 \|p_\infty-x_n^{j}\|^2 {\sum_{k \in [d]} (\eta^{j,k}_n)^2}/{d} \\
		&+ 2(1-\bar{\gamma}^1)\bar{\gamma}^2 {\|p_\infty-x_n^{j}\|^2}({ \sum_{k \in [d]} |\eta^{j,k}_n|^2 }/d)^\frac{1}{2}
		=: Y^j_n \|x^j_n - p_\infty\|^2.
	\end{aligned}\end{align}
	Observe that $\mathbb{E}[(\eta^{j,k}_n)^2]=1$ and 
	\[
		\mathbb{E}\Big[ ( \sum_{k \in [d]} |\eta^{j,k}_n|^2 )^\frac{1}{2} \Big] = \sqrt{2}\Gamma({d+1}/{2})/\Gamma({d}/{2}) < \sqrt{d},
	\]
	from the second moment of the Chi-square distribution. Therefore we have
	\[
		\mathbb{E}[ Y^j_n ] = (1 - \bar{\gamma}^1 + \bar{\gamma}^2)^2 - 2(1-\bar{\gamma}^1)\bar{\gamma}^2\left( 1 - {\sqrt{2/d}\Gamma({d+1}/{2})}/{\Gamma({d}/{2})} \right) < 1,
	\]
	where $Y^j_n$ is defined in \eqref{B4}. Then we use a similar procedure to find the desired result. 
\end{proof}

\begin{remark}\label{R3.3}
	The estimate in \eqref{B3-1} states that the convergence rate of $\|x^i_n-x^j_n\|$ is exponential for all but finitely may $n$ \emph{after $f(p_n)$ ceases to improve}. Therefore, if some agents locate the global minimizer, it can be inferred that the algorithm terminates soon.
\end{remark}

Theorem \ref{T3.3} states that inability to improving the objective is a sufficient condition for a consensus. Observe that this is not a necessary condition. To this end, the emergence of a consensus state is almost characterized by the vanishing of the change in the current best position.

\begin{proposition}[Consensus criterion]\label{P3.4} Let $f$ be a measurable function and $(x^i_n)$ be a sequence of iterates generated by Algorithm~\ref{Alg1}. Suppose that \eqref{B2} holds or $(1 - \bar{\gamma}^1 + \bar{\gamma}^2)^2 \leq \frac{1}{2}$. Then for almost every $\omega$, a sample path $X_n(\omega) = (x_n^1(\omega),\ldots,x_n^N(\omega))$ exhibits a consensus if and only if $\lim_{n \to \infty}\|p_n(\omega)-p_{n+1}(\omega)\| = 0$.
\end{proposition}

\begin{proof}
We will assume that $\mathcal{P}_S$ does not exist, since the proof under its existence naturally follows from the fact that projection map is non-expansive for each coordinate. Obviously, consensus implies $\|p_n-p_{n+1}\| \to 0$ because $p_n, p_{n+1} \in \{ x_{n+1}^i \mid i \in [N] \}$ for each $n$. Note that $x^i_{n+1}=x^i_n$ if $x^i_n=p_n$. To prove the converse, let $\Omega'$ be the event $\|p_n-p_{n+1}\| \to 0$. If $\mathbb{P}[\Omega']=0$, then there is nothing to prove. Therefore we may assume $\mathbb{P}[\Omega']>0$. First consider the case $F^i=F_1$. Similar to the proof of Theorem \ref{T3.3}, we have
\[
	x^{i,k}_{n+1} - p^k_{n+1} = (1 - \gamma^1 - \gamma^2\eta^{i,k}_n)(x^{i,k}_n - p^k_{n}) + (p_n^k - p_{n+1}^k).
\]
Take the absolute value on both sides and apply the triangle inequality to obtain
\[
	|x^{i,k}_{n+1} - p^k_{n+1}|
	\leq |1 - \gamma^1 - \gamma^2\eta^{i,k}_n||x^{i,k}_n - p^k_n| + |p_n^k - p_{n+1}^k|
	= X^{i,k}_n|x^{i,k}_n - p^k_n| + |p_n^k - p_{n+1}^k|,
\]
where $X_n^{i,k}$'s are defined in the proof of Theorem \ref{T3.3}. Then by induction,
\begin{align}\label{B4-0-1}
	|x^{i,k}_{n} - p^k_{n}|
	\leq |x^{i,k}_{0} - p^k_{0}|\prod_{I=0}^{n-1} X^{i,k}_I 
	+ \sum_{J=0}^{n-1} |p_J^k - p_{J+1}^k|\prod_{K = J+1}^{n-1} X^{i,k}_K,
\end{align}
where $\prod_{K = n}^{n-1} X^{i,k}_K = 1$. The same argument that employed in \eqref{B3-1} ensures that $\prod_{I=m}^{n-1} X^{i,k}_I$ converges to zero almost surely for each $m \in \{0\} \cup \mathbb{N}$. Furthermore, for $Z^{i,k}_{n-1}:=\sum_{J=0}^{n-1}\prod_{K = J+1}^{n-1} X^{i,k}_K$, observe that $Z^{i,k}_n = X^{i,k}_nZ^{i,k}_{n-1}+1$ to see
\[
	\mathbb{E}[Z^{i,k}_{n+1} - 1/(1-\mu)  \mid \mathcal{F}_n ] = \mu\left( Z^{i,k}_n - 1/(1-\mu) \right), \quad \mu = \mathbb{E}[X_n^{i,k}],
\]
where $\mathcal{F}_n$ is a $\sigma$-algebra generated by $\eta_m^j$ for $m \in [n] \cup \{0\}$ and $j \in [N]$. From \eqref{B2}, for $\mu \in (0,1)$, $Z^{i,k}_n - 1/(1-\mu)$ is a super-martingale. Since the expectation of $Z^{i,k}_n$ is uniformly bounded above by $\sum_{n=0}^\infty \mu^n = 1/(1-\mu)$, the limit of $Z_n^{i,k}$ is well defined and almost surely finite from the martingale convergence theorem. Now fix $\omega \in \Omega'$, so that $\|p_m-p_{m+1}\| < \varepsilon$ for any $\varepsilon>0$ whenever $m > M_\varepsilon$. From \eqref{B4-0-1},
\begin{align*}
	|x^{i,k}_{n} - p^k_{n}|
	\leq& \underbrace{|x^{i,k}_{0} - p^k_{0}|\prod_{I=0}^{n-1} X^{i,k}_I}_{\to ~ 0 ~ \text{a.s.}}  \\
	&+ \underbrace{\sum_{J=0}^{M_\varepsilon} |p_J^k - p_{J+1}^k|\prod_{K = J+1}^{n-1} X^{i,k}_K}_{\to ~ 0 ~ \text{a.s.}}
	+ \underbrace{\sum_{J=M_\varepsilon+1}^{n-1} |p_J^k - p_{J+1}^k|\prod_{K = J+1}^{n-1} X^{i,k}_K}_{\leq ~ \varepsilon Z_{n-1}^{i,k}}.
\end{align*}
Since the choice of $\varepsilon > 0$ is arbitrary and $\lim_{n \to \infty} Z_{n}^{i,k}$ exists almost surely in $\bbr_+$, we conclude $|x^{i,k}_{n} - p^k_{n}| \to 0$ almost surely. We apply this argument to each $i,k$ to deduce the desired result when $F^i=F_1$. \newline
\indent Now we consider the case $F^j=F_2$. In this case obtain
\[
	x^{j,k}_{n+1} - p_{n+1}^k = (1-\bar{\gamma}^1)(x^{j,k}_n - p_n^k) + \bar{\gamma}^2|p_n - x^j_n|{\eta_{n,k}^j}/{\sqrt{d}} + (p_n^k-p_{n+1}^k).
\]
From calculations similar to those used to derive \eqref{B4} and the fact that $(a+b)^2 \leq 2a^2+2b^2$, we have
\[
	\|x_{n+1}^{j} - p_{n+1}\|^2 \leq 2Y_n^j\|x_n^j-p_n\|^2 + 2\|p_n-p_{n+1}\|^2,
\]
where $Y_n^j$ is defined in \eqref{B4}. From the assumption $(1 - \bar{\gamma}^1 + \bar{\gamma}^2)^2 \leq \frac{1}{2}$, we have $\mathbb{E}[Y_n^j] < 1/2$ and this completes the proof by the same argument.
\end{proof}

Proposition \ref{P3.4} reveals that in order to ensure a consensus, it suffices to examine the behavior of the sequence $(p_n)$, whereas Proposition \ref{P3.2} indicates that $p_n$ is constrained in some sub-level set. Combining these insights, we deduce a dichotomy: irrespective of the progress of $(f(p_n))$, agents either exhibit a consensus or attain the global minimum.

\begin{theorem}[Boundedness implies consensus or global optimization]\label{T3.5} Let $(x^i_n)$ be a sequence of iterates generated by Algorithm~\ref{Alg1} with initial data $(x^i_0)$ under the fully isotropic diffusion (i.e. $F^j=F_2$ for each $j \in [N]$) and the parameters satisfy $(1 - \bar{\gamma}^1 + \bar{\gamma}^2)^2 \leq \frac{1}{2}$. Suppose that $f$ is Borel measurable and let $f^*$ be an essential infimum of $f$ with respect to the Lebesgue measure. Assume that the sub-level set $\{ x \in S \mid f(x) \leq \min_{i \in [N]} f(x_0^i) \}$ is bounded. Then for almost every sample path, either $\lim_{n \to \infty} f(p_n) \leq f^*$ or the agents exhibit a consensus. In particular, if $f$ is continuous on $S$ with a unique global minimizer $x^*$ and $f^*=f(x^*)$ holds, then a consensus emerges almost surely.
\end{theorem}

\begin{proof}
We again focus on the case that $\mathcal{P}_S$ does not exist. First assume that $f^* \in \bbr$. Throughout the proof, we define $\tilde{\mathbb{P}}$ as a probability measure conditioned on the given initial position that some agent lies in a bounded sub-level set:
\[
	\tilde{\mathbb{P}}[\, \cdot \,] := \mathbb{P}[\, \cdot \mid x^i_0, ~ i \in [N]], ~ \text{where} ~ S_0:=\{ x \in S \mid f(x) \leq \min_{i \in [N]} f(x_0^i) \} ~ \text{is bounded}.
\] Then from the disintegration theorem, it suffices to prove Theorem \ref{T3.5} for $\tilde{\mathbb{P}}$. Recall that the noises are external, so that the distribution of the nosies are invariant under the choice of initial data. Recall that $S_\varepsilon^f = \{ x \in S \mid f^* \leq f^* + \varepsilon \}$. \newline
\indent We claim that for some positive constant $D_1$,
\begin{align}\begin{aligned}\label{B4-1}
	\tilde{\mathbb{P}}&[ (p_{n-1} \notin S_\varepsilon^f) \wedge (p_n \in S_\varepsilon^f)  ] \geq D_1 \tilde{\mathbb{P}}[\bar{\Omega}_n], \\
	&\text{where} ~ \bar{\Omega}_n :=\{ \omega \in \Omega \mid p_{n-1}(\omega) \notin S_\varepsilon^f \} \cap \{ \omega \in \Omega \mid \|p_{n-2}-p_{n-1}\| \geq \varepsilon\}.
\end{aligned}\end{align}
To this end, take any $\omega \in \bar{\Omega}_n$. Since $f$ is Borel measurable, $S_\varepsilon^f$ is Borel measurable and has a positive Lebesgue measure from the definition of the essential infimum. Now, choose $j=j(\omega)$ to be a random index satisfying $p_{n-2} = x_{n-2}^j$. Then we have
\begin{align*}
	&x_{n-1}^j = x_{n-2}^j = p_{n-2} \in S_0, \\
	&\|x_{n-1}^j - p_{n-1}\|=\|x_{n-2}^j - p_{n-1}\|=\|p_{n-2} - p_{n-1}\| > \varepsilon, \\
	&x_n^{j} \in S^f_\varepsilon ~ \Longleftrightarrow ~
	~ p_{n-2} + \bar{\gamma}^1(p_{n-1} - p_{n-2}) + \bar{\gamma}^2 \|p_{n-1}-p_{n-2}\| {\eta_{n-1}^{j}}/{\sqrt{d}} \in S_\varepsilon^f.
\end{align*}
Since each $p_n$ is contained in the bounded set $S_0$ and $\bar{\gamma}^1 \in (0,1)$, $p_{n-2} + \bar{\gamma}^1(p_{n-1} - p_{n-2})$ is contained in $\mathrm{conv}(S_0)$, the closed convex hull generated by $S_0$. Let $D_2$ be a diameter of this convex hull. Note that $D_2$ is finite from the boundedness of $S_0$. In other words, for each $\omega \in \bar{\Omega}_n,p_{n-2}$ and $p_{n-1}$,
\[
	x_n^{j} \sim \mathcal{N}_d(\underbrace{p_{n-2} + \bar{\gamma}^1(p_{n-1} - p_{n-2})}_{\in \mathrm{conv}(S_0)} , \underbrace{(\bar{\gamma}^2 \|p_{n-1}-p_{n-2}\|/{\sqrt{d}})^2I_d}_{\in [(\bar{\gamma}^2\varepsilon)^2/d,(\bar{\gamma}^2D_2)^2/d]} ).
\]
Take any $\mu \in \mathrm{conv}(S_0)$, $\sigma \in [\bar{\gamma}^2\varepsilon/\sqrt{d},\bar{\gamma}^2D_2/\sqrt{d}]$ and suppose that $\zeta_{\mu,\sigma}$ is drawn from $\mathcal{N}_d(\mu,\sigma^2 I_d)$. Let $\mathrm{Leb}$ and $f_{\mu,\sigma}$ refer to the Lebesgue measure and the probability density function of $\mathcal{N}_d(\mu,\sigma I_d)$, respectively. Now consider
\begin{align*}
	&m_{\mu,\sigma}:=\mathrm{Leb}(S_\varepsilon^f)\min_{x \in \mathrm{conv}(S_0)}f_{\mu,\sigma}(x), \\
	&m := \inf\{ m_{\mu,\sigma} \mid \mu \in \mathrm{conv}(S_0), ~ \sigma \in [\bar{\gamma}^2\varepsilon/\sqrt{d},\bar{\gamma}^2D_2/\sqrt{d}]\}.
\end{align*}
Since $S_\varepsilon^f \subset S_0$, the probability of $\zeta_{\mu,\sigma} \in S_\varepsilon^f$ is not less than $m_{\mu,\sigma}$, which is positive from the compactness of $\mathrm{conv}(S_0)$. Also, $m>0$ from the compactness of the domain of $\mu$ and $\sigma$. \newline
\indent In summary, for $\omega \in \bar{\Omega}_n$ and a suitable index $j=j(\omega) \in [N]$, we have $p_n \in S_\varepsilon^f$ whenever $\eta_{n-1}^{j} \in M$ for some random set $M = M(\omega)$. Note that for each $i \in [N]$, the probability of $\eta_{n-1}^i \in M$ is not less than $m$, which is independent of $\omega$. Since $j$ is a random index, instead of specifying $j$, we consider the scenario such that $\eta_{n-1}^i \in M$ for each $i \in [N]$, which has a probability not less than $m^N$. Then the disintegration theorem yields
\[
	\tilde{\mathbb{P}}[ (p_{n-1} \notin S_\varepsilon^f) \wedge (p_n \in S_\varepsilon^f)  ] \geq m^N\tilde{\mathbb{P}}[\bar{\Omega}_n] =: D_1\tilde{\mathbb{P}}[\bar{\Omega}_n],
\]
where $N$ is the number of agents. This proves the claim \eqref{B4-1}. \newline
\indent Now we utilize Proposition \ref{P3.2} to observe
 \begin{align*}
{\small
 	\tilde{\mathbb{P}}[\lim_{n \to \infty}f(p_n) \leq f^* + \varepsilon ]
	\geq \sum_{n=2}^\infty \tilde{\mathbb{P}}[ (p_{n-1} \notin S_\varepsilon^f) \wedge (p_n \in S_\varepsilon^f)  ]
	\geq D_1 \sum_{n=2}^\infty \tilde{\mathbb{P}}[\bar{\Omega}_n].
}
 \end{align*}
Since the leftmost side is at most 1, the Borel-Cantelli lemma implies
\[
	\tilde{\mathbb{P}}[ (p_{n-1} \in S_\varepsilon^f) \vee (\|p_{n-2}-p_{n-1}\| < \varepsilon) ~ \text{eventually} ] = 1.
\]
If $p_{n-1} \in S_\varepsilon^f$ for some $n$, then $\lim_{n \to \infty}f(p_n) \leq f^* + \varepsilon$. If $\|p_{n-2}-p_{n-1}\| < \varepsilon$ eventually, then $\limsup_{n \to \infty}\|p_{n-2}-p_{n-1}\| \leq \varepsilon$. Therefore
\[
	\tilde{\mathbb{P}}[ (\lim_{n \to \infty}f(p_n) \leq f^* + \varepsilon) \vee (\limsup_{n \to \infty}\|p_{n-2}-p_{n-1}\| \leq \varepsilon) ] = 1.
\]
Since the choice of $\varepsilon > 0$ is arbitrary, the continuity of $\tilde{\mathbb{P}}$ leads to
\[
	\tilde{\mathbb{P}}[ (\lim_{n \to \infty}f(p_n) \leq f^*) \vee (\lim_{n \to \infty}\|p_{n-2}-p_{n-1}\| = 0) ] = 1.
\]
In particular, if $f$ is continuous on $S$ with a unique global minimizer $x^*$ and $f^{-1}(f^*)=\{x^*\}$, then $f(p_n) \to f^*$ implies $p_n \to x^*$. This leads to $\|p_{n-2} - p_{n-1}\| \to 0$. Since the vanishing of $\|p_{n-2}-p_{n-1}\|$ is equivalent to a consensus in the sense of Proposition \ref{P3.4}, we have the desired result. When $f^* = -\infty$, the proof is similar if we replace $S_\varepsilon^f$ to $\{ x \in S \mid f(x) < -\varepsilon^{-1} \}$. Note that $f$ cannot be continuous in this case.
\end{proof}

\section{Global optimization}\label{sec:4}
In this section, we investigate the conditions under which agents perform global optimization. Then, based on Proposition \ref{P3.2}, we introduce the DCBO with restart to increase the possibility of global optimization.

\subsection{Conditions for global optimization}
We first show that under a mild condition, Algorithm \ref{Alg1} attains the global minimum of $f$ in the best scenario.
\begin{proposition}\label{P4.1}
	Let $(x^i_n)$ be a sequence of iterates generated by Algorithm \ref{Alg1}, where $x^i_0 \overset{\mathrm{i.i.d.}}{\sim} \rho_{\mathrm{in}}$. Suppose that $f$ is continuous on $S$ and its global minimizer exists in the support of $\rho_{\mathrm{in}}$. Then for $f_\infty := \lim_{n \to \infty}f(p_n)$, we have
	\[
		\essinf_{\omega \in \Omega} f_\infty(\omega) = \min f.
	\]
\end{proposition}
\begin{proof}
Let $x^*$ be a global minimizer contained in the support of $\rho_{\mathrm{in}}$. Since $f$ is continuous on $S$, for any $\varepsilon>0$, $\{ x \in S \mid f(x) < f(x^*) + \varepsilon \}$ is an open set in $S$ containing $x^*$. As $x^*$ is contained in the support of $\rho_{\mathrm{in}}$, $\mathbb{P}[ f(x^i_0) < f(x^*) + \varepsilon]>0$. Now utilize $f(x^i_0) \geq f(p_0) \geq f(p_n) \geq \min f$ to obtain
\[
	\min f \leq \essinf_{\omega \in \Omega} f_\infty(\omega) \leq \essinf_{\omega \in \Omega} f(p_n(\omega)) \leq \essinf_{x \in \mathrm{supp}(\rho_\mathrm{in})}f(x) = f(x^*) = \min f.
\]
\end{proof}

\begin{remark}
For a discretized CBO with anisotropic and homogeneous noises, the work \cite{HJK} asserts that when the agents converge to some random vector $x_\infty$, 
	\[
		\essinf_{\omega \in \Omega} f(x_\infty(\omega)) \leq \min f + E,
	\]
where $E>0$ is a positive constant. Proposition \ref{P4.1} corresponds to the case $E=0$, thereby implying that DCBO improves the cited result from the worst-case analysis perspective.
\end{remark}

A natural question then arises: under what conditions does the global minimum occur? The following simple proposition indicates a condition for global optimization. Recall that $f^*$ is the essential infimum of $f$ with respect to the Lebesgue measure, and $S_\varepsilon^f = \{ x \in S \mid f(x) \leq f^* + \varepsilon \}$. From the definition of $f^*$, $S_\varepsilon^f$ is nonempty and has a positive measure for any $\varepsilon > 0$. 

\begin{proposition}\label{P4.2}
	Let $(x^i_n)$ be a sequence of iterates generated by Algorithm \ref{Alg1}. Suppose that $f^* \in \bbr$. Then the following two statements are equivalent.
	\begin{enumerate}
		\item[(1)] $f(p_n) < f^* + \varepsilon$ for all but finitely many $n \in \mathbb{N}$. 
		\item[(2)] $x^i_n \in S^f_\varepsilon$ for some $i \in [N], n \in \mathbb{N}$.
	\end{enumerate}
\end{proposition}

\begin{proof}
	Suppose that (1) holds, so that $f(p_n) < f^* + \varepsilon$ for some $n$. Then for the index $j$ satisfying $x^j_n = p_n$, we have $x^j_n \in S_\varepsilon^f$, which implies (2). If (2) holds, then we obtain
	\[
			f^* + \varepsilon > f(x^i_n) \geq f(p_n) \geq f(p_m), \quad \forall m \geq n.
	\]
\end{proof}

Therefore, estimating the probability for global optimization is based on the likelihood that any agent is attracted to the sub-level set $S_\varepsilon^f$. Given $f$, the transition
\[	y_n := (x^1_n,x^2_n,\cdots,x^N_n) \mapsto (x^1_{n+1},x^2_{n+1},\cdots,x^N_{n+1}) =: y_{n+1}\]
is memoryless. Hence, treating $\mathbb{R}^{Nd}$ as a state space, $y_n \mapsto y_{n+1}$ forms a Markov chain for the measurable function $f$. To estimate the probability for optimization, define the set $A^f_\varepsilon \subset \mathbb{R}^{Nd}$ as
\[	y  = (x^1, \cdots, x^N) \in A^f_\varepsilon \quad \Leftrightarrow \quad x^i \in S^f_\varepsilon ~ \text{for some $i \in [N]$}.
\]
Then, since $A^f_\varepsilon$ is an absorbing set from Proposition \ref{P3.2}, for the transition probability kernel $P$ of the chain and the initial measure $\nu$ corresponding to $\rho_{\mathrm{in}}$ we have
\begin{align*}
	&\mathbb{P}[\text{ $x^i_n \in S^f_\varepsilon$ for some $i \in [N], n \in \mathbb{N}$ }] \\
	&\hspace{0.5cm}= \int_{y_0 \in A^f_\varepsilon} \nu(dy_0)
	+ \sum_{n=1}^\infty \int_{y_0 \notin  A^f_\varepsilon} \cdots \int_{y_{n-1} \notin  A^f_\varepsilon}
	\nu(dy_0)P(y_0,dy_1)\cdots P(dy_{n-1},A^f_\varepsilon).
\end{align*}
In principle, this probability can be estimated if we can estimate $P$ and $A_\varepsilon^f$. However, the transition kernel depends on the definition of $p_n$, which in turn relies on the structure of $f$. Since performing this calculation for a general $f$ is expected to be extremely challenging, we extract the parts that do not depend on $P$. Then we obtain
\begin{align*}
	\mathbb{P}[\text{ $x^i_n \in S^f_\varepsilon$ for some $i \in [N], n \in \mathbb{N}$ }]
	\geq \int_{y_0 \in A^f_\varepsilon} \nu(dy_0) = 1 - (1 -  \rho_{\mathrm{in}}(S_\varepsilon^f))^N.
\end{align*}
In other words, we consider a random search problem where at least one of the initial positions is included in $S_\varepsilon^f$. To this end, suppose that for any global minimizer $x^*$, there exist functions $f_m$, $f_M$ satisfying
\begin{align}\label{XX1}
	f_m(x-x^*) \leq f(x) - f(x^*) \leq f_M(x-x^*), \quad \text{with } f_m(0)=f_M(0)=0.
\end{align}
Then that size of $S^f_\varepsilon$ can be estimated:
\begin{align*}
	S_\varepsilon^M(x^*) := \{ x \in S \mid f^M(x-x^*) < \varepsilon \}
	&\subset S_\varepsilon^f := \{ x \in S \mid f(x) < f^* + \varepsilon \} \\
	&\subset S_\varepsilon^m(x^*) := \{ x \in S \mid f^m(x-x^*) < \varepsilon \}.
\end{align*}
This observation leads to the following result.
\begin{proposition}\label{P4.3}
	Let $(x^i_n)$ be a sequence of iterates generated by Algorithm \ref{Alg1}, where $x^i_0 \overset{\mathrm{i.i.d.}}{\sim} \rho_{\mathrm{in}}$, and $S_0^f$ be a set of global minimizers. Suppose that $f$ is continuous on $S$ and its global minimizer exists in the support of $\rho_{\mathrm{in}}$. Then for $f_\infty = \lim_{n \to \infty}f(p_n)$, we have
		\begin{align}\label{X3}
			f_\infty - \min f < \varepsilon,
		\end{align}
		with a probability not less than $1 - (1 -  \rho_{\mathrm{in}}(S_\varepsilon^M))^N$. In particular, if $N \to \infty$, then this probability converges to 1. If $f$ has a non-empty bounded sub-level set, then $\limsup_{n \to \infty}\mathrm{dist}(S_0^f,p_n) \leq \mathrm{diam}(S_\varepsilon^m)$ in case that \eqref{X3} occurs.
\end{proposition}

\begin{remark}
	In the mean-field CBO models, the existence of $f_m$, and sometimes $f_M$, is commonly assumed for convergence analysis. For instance, \cite{FKR} takes $f_M = K_1(1+\|x-x^*\|^{K_2})\|x-x^*\|$ and $f_m(x) = K_3\|x\|^{1/K_4}$ in the vicinity of the global minimizer for $K_1, K_3, K_4 > 0$ and $K_2 \geq 0$, resulting in a {\L}ojasiewicz-type inequality:
\[
	K_3\|x-x^*\|^{1/K_4} \leq f(x) - f(x^*) \leq K_1(1+\|x-x^*\|^{K_2})\|x-x^*\|.
\]
\end{remark}

In other word, if information about $f_m$ and $f_M$ are given, then a lower bound of the probability of global optimization can be obtained in terms of the initial position. If we have a quantitative measure of the continuity of $f$, then the size of $S_\varepsilon^f$ can be stated more explicitly.

\begin{definition}[Modulus of continuity]
	A real-valued function $f$ admits a \emph{local modulus of continuity} $k : [0,\infty] \to [0,\infty]$ at $x$ if $k$ is increasing and
	\begin{align*}\label{mod}
		|f(x) - f(y)| \leq k(\|x-y\|) \quad \text{and} \quad \lim_{x \to 0}k(x) = k(0) = 0,
	\end{align*}
	for all $y \in \mathrm{dom}(f)$. In this case, $f$ is called $k$-continuous at $x$. If $k$ can be chosen independently of $x$, then $f$ admits a \emph{modulus of continuity} $k$ and $f$ is called $k$-continuous.
\end{definition}

Therefore, $f_M$ can be specified from $k$ if $f_M$ is $k$-continuous at $x^*$, where $k$ depends on the regularity of $f$. 

\begin{example}\label{E1} If $f$(resp. $g$) is $k$(resp. $k'$)-continuous, then $f \circ g$ and $\alpha f + \beta g$ are $(k \circ k')$-continuous and $(|\alpha|k+|\beta|k')$-continuous respectively, where $\alpha$ and $\beta$ are constants. Therefore, the Ackley function
\[
	F_{Ack}(x) := -a\exp\left(-b\|x\|/\sqrt{d}\right)-\exp\left(\sum_{i=1}^{d}\cos(cx_i)/d\right)+a+\exp(1)
\] 
has the following modulus of continuity $k_{Ack}$:
		\[
			k_{Ack}(t) = a( 1 - \exp(-bt/\sqrt{d})  ) + \exp(1)\left(1 - \exp(-ct/\sqrt{d}) \right).
		\]
		Since $F_{Ack}$ has the global minimum at the origin $x^*=0$, one can choose  $f_M(x)$ as $k_{Ack}(\|x\|)$. If we draw each $x^i_0$ from $\mathrm{Unif}([-L,L]^d)$, then with a probability at least
\begin{align*}
	1 - \left( 1 - \frac{\pi^{d/2}(k^{-1}_{Ack}(\varepsilon))}{L^d\Gamma(d/2+1)}\right)^N,
\end{align*}
we have $f_\infty - f(x^*) < \varepsilon$ for small $\varepsilon>0$. 
Then, choosing $f_m := -a\exp\left(-b\|x\|/\sqrt{d}\right)+a$ yields
\[
	\limsup_{n \to \infty}\| x^* - p_n \| < 2{\sqrt{d}}/{b}\log({a}/{(a-\varepsilon)}).
\] 
\end{example}

 \subsection{DCBO with restart} 
Consider the scenario where we have a candidate for a local minimizer, say $p$. Proposition \ref{P3.2} asserts that if Algorithm \ref{Alg1} is run with the initial positions of the agents including $p$, then the resulting sequence $(p_n)$ is no worse than $p$ in minimizing $f$. This inspires the application of Algorithm \ref{Alg1} with restart. The efficacy of restarted DCBO is supported by the following proposition.

\begin{proposition}\label{P4.7}
    Let $f$ be any Borel measurable function. For each $m \in \mathbb{N}$, choose any $M_m \in \mathbb{N}$. Apply Algorithm \ref{Alg1} repeatedly as follows:
    \begin{enumerate}
    \item Draw $x^i_0 \overset{\mathrm{i.i.d.}}{\sim} \rho_{\mathrm{in}}$ for $i \in [N]$.
    \item Define $p^m$ as follows.
    \begin{enumerate}
        \item[(a)] If $m=1$, then $p^m = x^i_0$.
        \item[(b)] If $m \neq 1$, then set $p^m = p_{M_m}$, where $p_{M_m}$ is a copy of $p_n$ at the $M_m$-th iteration in the $(m-1)$th round.
    \end{enumerate}
    \item Run the $m$th round of Algorithm \ref{Alg1} with initial positions $(p^m,x^2_0,\ldots,x^N_0)$.
    \end{enumerate}
    Then, we have
    \begin{align*}
        &f(p^0) \geq f(p^1) \geq f(p^2) \geq \cdots \geq f(p^n) \geq \cdots, \text{ and} \\
        &\lim_{n \to \infty} f(p^n) \leq \essinf_{x \in \mathrm{supp}(\rho_\mathrm{in})} f(x) =: f^*_{\rho_\mathrm{in}},
    \end{align*}
    almost surely. If we further assume that 
    \begin{enumerate}
    	\item $f$ is continuous on $S$,
        \item $f$ attains its global minimum, and
        \item $\mathrm{supp}(\rho_{\mathrm{in}}) = S$,
    \end{enumerate}
    then $\lim_{n \to \infty} f(p^n) = \min f$ almost surely.
\end{proposition}
\begin{proof}
    From Proposition \ref{P3.2}, $f(p^n)$ is monotonically decreasing in $n$, and its limit always exists in $\bbr \cup \{-\infty\}$. First, assume $f^*_{\rho_\mathrm{in}} \in \bbr$. Then for any positive constant $\varepsilon > 0$, we have $\rho_\mathrm{in}( \{ x \mid f(x) \leq f^*_{\rho_\mathrm{in}} + \varepsilon \} ) > 0$. Therefore, with probability 1, there exists $m' \in \mathbb{N}$ such that $x^i_0 \in \{ x \mid f(x) \leq f^*_{\rho_\mathrm{in}} + \varepsilon \}$ for some $i$ at the $m'$-th round. Thus,
    \[
        \lim_{n \to \infty} f(p^n) \leq f^*_{\rho_{\mathrm{in}}} + \varepsilon, \quad \text{a.s.}
    \]
    Since the choice of $\varepsilon$ is arbitrary, we obtain the desired result. When $f^*_{\rho_\mathrm{in}} = - \infty$, the proof is similar if we replace $\{ x \mid f(x) \leq f^*_{\rho_\mathrm{in}} + \varepsilon \}$ to $\{ x \in S \mid f(x) < -\varepsilon^{-1} \}$. In particular, when $f$ satisfies the additional assumptions, we have $f^*_{\rho_\mathrm{in}} = \min f \in \bbr$.
\end{proof}

\begin{remark}\label{R4.8}
To demonstrate the result in Proposition \ref{P4.7}, we optimized the $100$-dimensional Ackley function through 30 rounds of Algorithm using $N=50$ agents. The objective function has the global minimum of $0$ at the origin. Fig. \ref{fig:1a} illustrates the monotonic decrease of $f(p_n)$, and global optimization is achieved within a sufficient number of restarts. Furthermore, Fig. \ref{fig1} presents some insights on the convergence rate. Recall that the proof of Proposition~\ref{P3.4} suggests that $\|x^i - p^i\|$ (and hence $\|x^i-x^j\|$) can be controlled by $\sum_{n} \|p_n - p_{n+1}\|$. In Fig. \ref{fig:1b}, the number of iterations are correlated strongly with the sum $\sum \|p_{n+1}-p_n\|$ and the algorithm continue if some agent updates $p_n$ and accumulates $\sum \|p_{n+1}-p_n\|$. We observe a significant drop in the iteration count when $p_n$ approaches the global minimum and when the accumulation of $\sum \|p_{n+1}-p_n\|$ ceases. Consequently, the convergence speed of our algorithm heavily relies on how early an agent can find a point sufficiently close to the optimum. For instance, in Fig. \ref{fig:1b}, a suitable $p_n$ was determined in the 23rd round, leading to a rapid progression of the algorithm with fewer iterations afterward.

\begin{figure}[h]
\hspace{.4cm}
	\subfloat[Best function value $f(p_n)$]{\label{fig:1a}\includegraphics[scale=0.388]{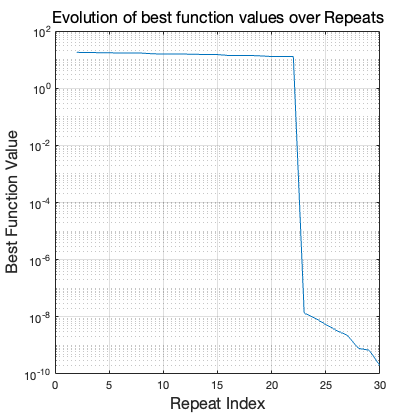}}
	\hfill
	\subfloat[The Number of iteration $I$ and $\sum_{n=1}^I \|p_n-p_{n-1}\|$]{\label{fig:1b}\includegraphics[scale=0.388]{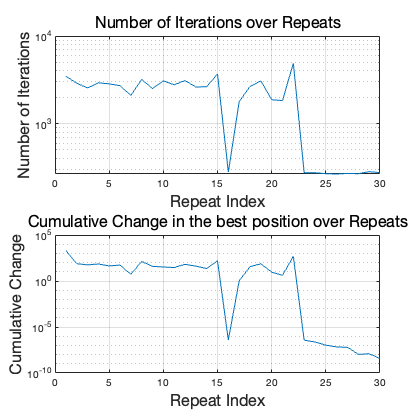}}
    \caption{DCBO with restart on 100-dimensional Ackley function. The repeat index counts the round number. Fig. \ref{fig:1a} depicts the value of the $f(p_n)$ at the last iteration for each repeat index. The top graph in Fig. \ref{fig:1b} represents the number of iterations in each repeat index, and the bottom graph displays the sum $\sum_{n}\|p_{n+1}-p_n\|$.}
    \label{fig1}
\end{figure}


\end{remark}

\section{Numerical experiments}\label{sec:5}
\setcounter{equation}{0}
In this section, we apply Algorithm \ref{Alg1} to a variety of problems and compare its performance with other methods including conventional CBO and the particle swarm optimization (PSO) models. The latter is another popular meta-heuristic method utilizing dynamics of multiple agents. The primary objective of this section is to see if Algorithm \ref{Alg1} is applicable to a variety of problems, relatively immune to various factors like dimensions, convexity and gradients of the objective function. Throughout the section, we use a mixture of diffusions; we pose $F^i = F_1$ for  $i \leq \lfloor N/2 \rfloor$, and $F^j = F_2$ for  $j > \lfloor N/2 \rfloor$.



\subsection{Benchmark}

We applied Algorithm \ref{Alg1} to various benchmark functions. We tested a total of 102 benchmark functions,\footnote{Available at https://www.sfu.ca/{\texttildelow}ssurjano/index.html} where the dimension of the inputs ranges from 2 to 80. For each function, we conducted 100 runs of \eqref{DCBO} and recorded the results. We used $\gamma^1 = 0.5, \gamma^2 = 1, \bar{\gamma}^1 = 0.4, \bar{\gamma}^2 = 0.7$. Note that although $(\bar{\gamma}^1,\bar{\gamma}^2)=(0.4,0.7)$ is not covered by \eqref{B1}, it is justified by \cite{FKR}, and indeed, we observed that consensus always occured. We compare the performance of Algorithm \ref{Alg1} with and without introducing restart against PSO and hmPSO \cite{CKTW}, respectively. The latter is known to enhance the performance of PSO by introducing appropriate heterogeneous perturbation. For the plain PSO and hmPSO, we used parameters $w = 0.729, c_1=c_2=1.5$, and for stochastic perturbation, we employed independent normal variates with mean 0 and standard deviation 0.005, as specified in \cite{CKTW}. Note that each benchmark problem can be solved with a high success rate by increasing the number of particles; however, our goal is to evaluate the relative performance of CBO, focusing on how it operates under limited resources rather than solving each problem perfectly.

\subsubsection{Experiment 1: Comparison between DCBO without restart and PSO}

We compare the performance of Algorithm \ref{Alg1} without restart to that of the plain PSO algorithm. For each objective function $f$ defined on a subset of $\bbr^d$, we executed up to $500 \times d$ iterations. We terminated DCBO when $\max_{i<j}\|x^i_n-x^j_n\|$ was less than $10^{-7}$. Both DCBO and PSO algorithms exhibited significantly different iteration counts depending on the objective function. In fact, the number of iterations differed significantly even for the same objective function (see Remark \ref{R4.8}). Thus it is difficult to establish a perfect parity in terms of iteration counts. However, since the main objective of this experiment is to verify the competitiveness of the CBO algorithm relative to PSO, the termination criteria for the PSO algorithm were set to be more stringent than conventional standards: we terminated PSO if the change in the best function value was less than $10^{-10}$ for 200 consecutive iterations before reaching the imposed maximum iteration count of $500 \times d$. This enables PSO to execute a greater number of iterations in the majority of cases. Specifically, for $N=50,100$, and $200$, the average number of iterations was recorded, and PSO had a higher count for 61, 63, and 68 functions, respectively.

\subsubsection{Experiment 2: ~ Comparison between DCBO with restart and hmPSO}

We compare the performance of the Algorithm \ref{Alg1} with restart to that of the hmPSO. For each objective function $f$ defined on a subset of $\bbr^d$, we consistently executed $500 \times d$ iterations. For the CBO, if uniform consensus did not occur within $100 \times d$ iterations, the algorithm was forced to restart to ensure that it was restarted at least 5 times. In contrast to Experiment 1, a uniform number of iterations is ensured.

\begin{table}[h]
\centering
\footnotesize
{
\begin{tabularx}{\textwidth}{l| >{\hsize=.85\hsize}X >{\hsize=1.2\hsize}X >{\hsize=.95\hsize}X| >{\hsize=.85\hsize}X >{\hsize=1.25\hsize}X >{\hsize=.9\hsize}X| >{\hsize=.85\hsize}X >{\hsize=1.25\hsize}X >{\hsize=.9\hsize}X}
\toprule
\textbf{$\#$ of agnets} & \textbf{} & $N$=50& \textbf{} & \textbf{} & $N$=100 & \textbf{} & \textbf{} & $N$=200 & \textbf{} \\
\midrule
\textbf{Statistics} & \textbf{Min} & \textbf{Mean} & \textbf{Med} & \textbf{Min} & \textbf{Mean} & \textbf{Med} & \textbf{Min} & \textbf{Mean} & \textbf{Med} \\
\midrule
DCBO $>$ PSO & 19 & 49 & 34 & 19 & 42 & 33 & 18 & 40 & 32 \\
DCBO $=$ PSO & 70 & 22 & 51 & 75 & 34 & 55 & 75 & 39 & 56 \\
DCBO $<$ PSO & 13 & 31 & 17 & 8 & 26 & 14 & 9 & 23 & 14 \\
\bottomrule
\end{tabularx}
}
\caption{Comparison between DCBO without restart and PSO.}
\label{tab1}
\end{table}

\begin{table}[h]
\centering
\footnotesize
{
\begin{tabularx}{\textwidth}{l| >{\hsize=.85\hsize}X >{\hsize=1.2\hsize}X >{\hsize=.95\hsize}X| >{\hsize=.85\hsize}X >{\hsize=1.25\hsize}X >{\hsize=.9\hsize}X| >{\hsize=.85\hsize}X >{\hsize=1.25\hsize}X >{\hsize=.9\hsize}X}
\toprule
\textbf{$\#$ of agnets} & \textbf{} & $N$=50& \textbf{} & \textbf{} & $N$=100 & \textbf{} & \textbf{} & $N$=200 & \textbf{} \\
\midrule
\textbf{Statistics} & \textbf{Min} & \textbf{Mean} & \textbf{Med} & \textbf{Min} & \textbf{Mean} & \textbf{Med} & \textbf{Min} & \textbf{Mean} & \textbf{Med} \\
\midrule
DCBO $>$ PSO & 21 & 41 & 25 & 17 & 39 & 25 & 17 & 38 & 26 \\
DCBO $=$ PSO & 74 & 52 & 67 & 80 & 57 & 70 & 83 & 60 & 72\\
DCBO $<$ PSO & 7 & 9 & 10 & 5 & 6 & 12 & 2 & 4 & 4 \\
\bottomrule
\end{tabularx}
}
\caption{Comparison between DCBO with restart and hmPSO.}
\label{tab2}
\end{table}

\begin{table}[h]
\centering
\footnotesize 
\begin{tabularx}{\textwidth}{@{}l>{\hsize=.7\hsize}X|XXX|>{\hsize=1.05\hsize}X>{\hsize=1.05\hsize}X>{\hsize=1.05\hsize}X@{}}
\toprule
\multicolumn{2}{c}{\textbf{}}  & \multicolumn{3}{c}{\textbf{DCBO without restart}} & \multicolumn{3}{c}{\textbf{PSO}} \\
\midrule
\textbf{} & \textbf{} & \textbf{$N=50$} & \textbf{$N=100$} & \textbf{$N=200$} & \textbf{$N=50$} & \textbf{$N=100$} & \textbf{$N=200$}\\
\midrule
Ackley & Min & 0 & 0 & 0 & 4.911 & 2.351 & 1.155 \\
 & Mean & 4.504 & 2.145 & 1.064 & 15.73 & 13.57 & 11.35 \\
 & Med & 0 & 0 & 0 & 16.03 & 14.55 & 13.02 \\
 & Iter & 3577 & 2983 & 2699 & 2875 & 2146 & 1733 \\
\midrule
Griewank & Min & 0 & 0 & 0 & 0.03687 & 0 & 0 \\
~ & Mean & 0.004187 & 0.004656 & 0.003203 & 201.5 & 182.4 & 153.5 \\
~ & Med & 0 & 0 & 0 & 180.9 & 180.7 & 180.1 \\
~ & Iter & 3617 & 3160 & 2819 & 3230 & 2141 & 1627 \\
\midrule
Rastrigin & Min & 201.0 & 84.57 & 25.87 & 378.3 & 354.4 & 317.5 \\
~ & Mean & 351.2 & 211.6 & 93.40 & 551.8 & 511.0 & 486.9 \\
~ & Med & 345.7 & 205.5 & 91.04 & 544.7 & 512.8 & 494.9 \\
~ & Iter & 2988 & 2994 & 2893 & 2789 & 2010 & 1573 \\
\midrule
Trid & Min & 21.25 & 1.692 & 2.307 & 5214000 & 9948000 & 2640000 \\
~ & Mean & 15690 & 9501 & 4380 & 60390000 & 67600000 & 63480000 \\
~ & Med & 13230 & 6717 & 3285 & 54130000 & 64370000 & 54240000 \\
~ & Iter & 40000 & 40000 & 40000 & 22210 & 18430 & 18890 \\
\midrule
Zakharov & Min & 0 & 0 & 0 & 104.3 & 25.03 & 0 \\
~ & Mean & 0 & 0 & 0 & 649.0 & 446.1 & 285.3 \\
~ & Med & 0 & 0 & 0 & 602.0 & 379.8 & 257.6 \\
~ & Iter & 7452 & 5626 & 4578 & 34770 & 31760 & 29270 \\
\midrule
Rosenbrock & Min & 3.998 & 0.03853 & 0.008562 & 166600 & 224900 & 55520 \\
~ & Mean & 52.72 & 44.00 & 36.18 & 689200 & 645100 & 620400 \\
~ & Med & 50.29 & 43.92 & 36.65 & 669000 & 637900 & 626000 \\
~ & Iter & 40000 & 40000 & 40000 & 12250 & 8827 & 13140 \\
\midrule
Powell & Min & 0.000017 & 0.000005 & 0.000001 & 551.4 & 716.9 & 482.6 \\
~ & Mean & 0.000023 & 0.000006 & 0.000002 & 8958 & 8724 & 7852 \\
~ & Med & 0.000024 & 0.000006 & 0.000002 & 8210 & 8472 & 7847 \\
~ & Iter & 40000 & 40000 & 40000 & 19340 & 18680 & 18570 \\
\midrule
Styblinski-Tang & Min & 212.0 & 127.2 & 28.26 & 339.3 & 282.7 & 296.9 \\
~ & Mean & 337.9 & 251.6 & 144.6 & 472.2 & 469.9 & 452.2 \\
~ & Med & 339.3 & 254.4 & 141.4 & 466.5 & 466.5 & 452.4 \\
~ & Iter & 2509 & 2384 & 2289 & 3151 & 2042 & 1547 \\
\bottomrule
\end{tabularx}
\caption{Comparison of high-dimensional ($d=80$) functions up to 4 significant digits. For each function, we recorded statistics for $f_\infty - \min f$ and the mean number of iterations. Each $f_\infty$ is obtained from the corresponding algorithm.}\label{tab3}
\end{table}

\begin{table}[h]
\centering
\footnotesize 
\begin{tabularx}{\textwidth}{@{}lX|XXX|XXX@{}}
\toprule
\multicolumn{2}{c}{\textbf{}}  & \multicolumn{3}{c}{\textbf{DCBO with restart}} & \multicolumn{3}{c}{\textbf{hmPSO}} \\
\midrule
\textbf{} & \textbf{} & \textbf{$N=50$} & \textbf{$N=100$} & \textbf{$N=200$} & \textbf{$N=50$} & \textbf{$N=100$} & \textbf{$N=200$}\\
\midrule
Ackley & Min & 0 & 0 & 0 & 1.646 & 1.228 & 0 \\
~ & Mean & 3.332 & 1.246 & 0.1691 & 12.34 & 6.639 & 0.9192 \\
~ & Median & 0 & 0 & 0 & 15.09 & 3.001 & 0 \\
\midrule
Griewank & Min & 0 & 0 & 0 & 0 & 0 & 0 \\
~ & Mean & 0.004409 & 0.004826 & 0.003868 & 0.007878 & 0.005906 & 0.01044 \\
~ & Median & 0 & 0 & 0 & 0 & 0 & 0 \\
\midrule
Rastrigin & Min & 35.82 & 1.990 & 0 & 390.2 & 366.1 & 304.6 \\
~ & Mean & 149.9 & 51.63 & 10.80 & 545.6 & 514.5 & 511.2 \\
~ & Median & 146.8 & 55.22 & 2.985 & 552.5 & 513.7 & 505.2 \\
\midrule
Trid & Min & 211.8 & 15.31 & 0.4107 & 506.7 & 420.0 & 1513 \\
~ & Mean & 17340 & 8900 & 5188 & 144200 & 160800 & 151400 \\
~ & Median & 13780 & 7456 & 4121 & 71910 & 81080 & 73210 \\
\midrule
Zakharov & Min & 0 & 0 & 0 & 0 & 0 & 0 \\
~ & Mean & 0 & 0 & 0 & 0 & 0 & 0 \\
~ & Median & 0 & 0 & 0 & 0 & 0 & 0 \\
\midrule
Rosenbrock & Min & 17.63 & 0.01050 & 0.009896 & 0.000037 & 0 & 0 \\
~ & Mean & 55.79 & 49.62 & 34.57 & 135.4 & 96.15 & 98.46 \\
~ & Median & 50.52 & 44.10 & 36.65 & 28.46 & 15.19 & 14.77 \\
\midrule
Powell & Min & 0.000019 & 0.000004 & 0.000001 & 118.8 & 43.65 & 75.18 \\
~ & Mean & 0.000025 & 0.000006 & 0.000002 & 352.1 & 394.0 & 413.3 \\
~ & Median & 0.000024 & 0.000006 & 0.000002 & 325.4 & 393.0 & 418.1 \\
\midrule
Styblinski-Tang & Min & 0 & 0 & 0 & 339.0 & 254.2 & 296.6 \\
~ & Mean & 68.71 & 5.079 & 0 & 459.5 & 441.1 & 447.0 \\
~ & Median & 70.39 & 0 & 0 & 452.1 & 437.9 & 452.1 \\
\bottomrule
\end{tabularx}
\caption{Comparison of high dimensional ($d=80$) functions with 40,000 iterations up to 4 significant digits. For each function, we recorded statistics for $f_\infty-\min f$, where $f_\infty$ is obtained from the corresponding algorithm.}\label{tab4}
\end{table}

\subsubsection{Results}
Tables \ref{tab1} and \ref{tab2} record which algorithm produced better results for each statistic (minimum, mean, median), rounded to six decimal places. If DCBO outperformed PSO, it is denoted as DCBO $>$ PSO. The other comparisons are analogously denoted. The cases DCBO $=$ PSO often corresponded to those in which both algorithms reached global optimum up to six decimal places. Overall, especially for high-dimensional functions, DCBO without restart (resp. with restart) performed better than plain PSO (resp. hmPSO), and DCBO with restart (resp. hmPSO) outperformed DCBO without restart (resp. PSO). Tables \ref{tab3} and \ref{tab4} present statistics and average iteration counts for several 80-dimensional functions. From Table \ref{tab3}, two observations regarding iteration counts emerge. First, for functions where global minimum was attained frequently (Ackley, Griewank, Zakharov), the algorithm terminates with fewer iterations as the number of agents $N$ increases. This trend likely stems from an early detection of a nearly optimal point (i.e., the emergence of $p_n \in S_\varepsilon^f$ for small $n$ and $\varepsilon$), as detailed in Remark \ref{R4.8}. As $N$ increases, $p_n$ has a higher probability of detecting a global minimizer, leading to faster convergence. Second, for some functions (Trid, Rosenbrock, Powell), a uniform consensus was not reached even after the maximum number of iterations (40,000). This tendency is attributed to the objective function being too flat to vanish $\|p_{n+1}-p_n\|$ in a short period. Additionally, while increasing $N$ in hmPSO did not always guarantee improved performance, DCBO with or without restart, typically achieved better outcomes with a higher value of $N$.


\subsection{Portfolio optimization}

We address the portfolio optimization problem and compare the results with those obtained using a conventional CBO model proposed in \cite{BHKLMY}. Based on Markowitz's portfolio theory \cite{Mar}, the goal here is to select an optimal portfolio weights $w = (w^1,\ldots,w^d)$ by minimizing the negative Sharpe ratio of the portfolio:
\begin{align}\begin{aligned}\label{Sharpe}
	\min_{w \in S} -\frac{w^\top \mu}{\sqrt{w^\top \Sigma w}}, ~
	S := \Big\{ w = (w^1,\ldots,w^d) \in \mathbb{R}^d \mid \sum_{i=1}^d w^i = 1, ~ w^i \geq 0, ~ i \in [d] \Big\},
\end{aligned}\end{align}
where $w$ and $\Sigma$ represent the mean vector and covariance matrix of the risky assets. Note that the domain $S$ is convex, and we apply $\mathcal{P}_S$ in Algorithm \ref{Alg1}. 

\indent We reused the Python code for \cite{BHKLMY}, given by the courtesy of the authors. Our dataset comprised six assets: Apple Inc., Microsoft Inc., Starbucks Inc., Tesla Inc., German Deutsche Bank, and gold, covering the period from January 2019 to November 2020. For comparison, we employed the discretized CBO model in \cite{BHKLMY}; used anisotropic and homogeneous noises and parameters $h=0.01, N=100, \lambda=0.5$, and $\sigma = 1$. To evaluate the impact of the $\beta$, we varied $\beta$ across $10, 10^2, 10^4$, and $10^6$. For DCBO, we maintained $N=100$ and used the parameters $(\gamma^1, \gamma^2, \bar{\gamma}^1, \bar{\gamma}^2) = (0.5,1,0.4,0.7)$ consistently. Each algorithm terminated when the maximum of $\|w^i_n-\bar{w}_n\|$ or $\|w^i_n-p_n\|$ fell below $10^{-5}$.

\begin{table}[h]
\footnotesize
\begin{tabular}{l|p{1.3cm}|p{1.3cm}|p{1.3cm}|p{1.3cm}|p{1.3cm}|p{1.3cm}}
\toprule
~ & $\beta = 10$ & $\beta = 10^2$ & $\beta = 10^4$ & $\beta = 10^6$ & DCBO & Glob. min\\
\midrule
Iteration & 1892.98 & 1925.71 & 1908.69 & 1896.54 & 74.29 & ~\\
Function value & -1.80051 & -1.95519 & -1.98411 & -1.98450 & -2.01450 & -2.01450\\
Distance & 0.339658 & 0.145540 & 0.109392 & 0.110149 & 0.000016 & ~\\
\bottomrule
\end{tabular}
\caption{Comparison between DCBO and anisotropic CBO with homogeneous noises.}
\label{portTable}
\end{table}

Table \ref{portTable} records the iteration count, function value, and distance from the global minimum $w^*$, which was computed by the sequential least squares programming (SLSQP) in Python. These values were averaged over 100 runs. As expected from Section \ref{sec:1}, the performance of the discretized CBO improved with increasing $\beta$ until the numerical precision was hit. Nevertheless, DCBO consistently outperformed the discretized CBO for all tested $\beta$ values, requiring significantly fewer iterations. This efficiency is attributed to the early detection of the global minimizer by using the heterogeneous noises and har-min operation, especially because problem \eqref{Sharpe} is moderately complex. Interestingly, DCBO's average function value was marginally lower than $-{(w^*)^\top \mu}/{\sqrt{(w^*)^\top \Sigma (w^*)}}$, while still adhering to the constraint $p_n \in S$. Hence, DCBO appears to be a viable alternative to SLSQP.

\subsection{Neural networks}

We explore a well-established high-dimensional problem in machine learning: training a convolutional neural network (CNN) classifier for the MNIST dataset of handwritten digits \cite{LCB}. This experiment follows the methodology outlined in \cite{R}, where the CNN architecture, parameterized by 2112 variables, is optimized using mini-batch techniques \cite{CJLZ,FHPS2}.\footnote{The code is available at https://github.com/KonstantinRiedl/CBOGlobalConvergenceAnalysis.} In \cite{R}, the authors compared a conventional CBO model with an enhanced CBO model incorporating memory effects and gradient information. Unfortunately we could not replicate the results from the code and setup provided in \cite{R} for the latter model. Hence our comparison is limited with the conventional CBO. We adhered to the same architecture and mini-batch strategy as in \cite{R}, modifying only the update rule to implement DCBO. The parameters used were $(\gamma_1,\gamma_2,\bar{\gamma}_1,\bar{\gamma}_2)=(0.5,0.5,0.4,0.5)$.

\begin{figure}[h]
\hspace{.4cm}
	\subfloat[Performance in 100 epochs]{\includegraphics[scale=0.36]{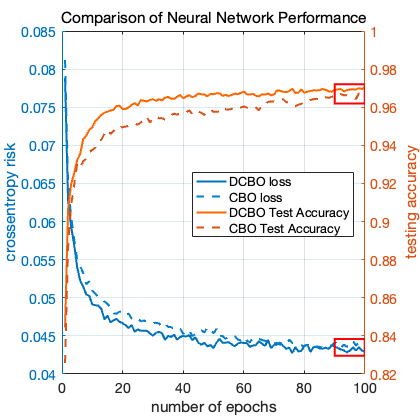}} \hspace{.3cm}
	\subfloat[Performance in the last 10 epochs]{\includegraphics[scale=0.36]{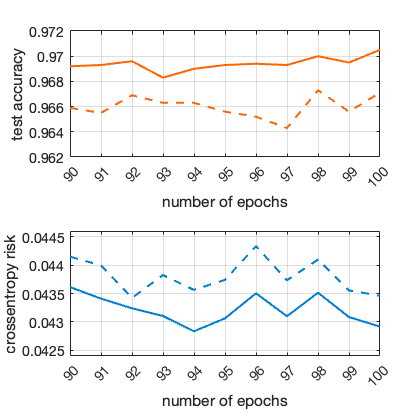}}
    \caption{MNIST Classification.}
    \label{fig7}
\end{figure}

Fig. \ref{fig7} displays the test accuracy and empirical risk for the CNN. Ultimately, DCBO and CBO achieved test accuracies of 0.9705 and 0.9671, and cross-entropy risks of 0.0429 and 0.0435, respectively. As seen in Fig. \ref{fig7}, DCBO demonstrated rapid performance improvements with fewer epochs and consistently outperformed the conventional CBO over the long term. Riedl \cite{R} employed a cooling strategy in each epoch to enhance performance: doubling $\beta$ and adjusting the diffusion parameter $\sigma$ to $\sigma/\log_2(E+2)$ in the $E$-th epoch. In contrast, DCBO achieved superior performance to the conventional CBO without the need for annealing. Although a direct comparison was not feasible due to implementation issues, compared to the results in \cite{R}, our findings suggest that DCBO may slightly outperform the memory effect or memory drift-based CBO model.

\subsection{Compressed sensing}
Compressed sensing problem concerns retrieving a $d$-dimensional signal $x$ from fewer measurements $b \in \bbr^m ~ (m < d) $, assuming that the signal is sparse. Let the number of non-zero elements in $x$ be $s$. The signal and measurements are related by $b=Ax$, where $A \in \mathbb{R}^{m \times d}$ is a sensing matrix. Our goal is to find $x$ given $A$ and $b$. This problem can be approached by solving
\begin{align}\label{CS}
	\min_{x \in \mathbb{R}^d} \frac{1}{2}\|Ax-b\|^2_2 \quad \text{subject to} \quad \|x\|_p \leq r,
\end{align}
for a suitable $p \in [0,1]$, where $\|x\|_p :=[(x^1)^p+\ldots+(x^d)^p]^{1/p}$ is the $\ell^p$ (quasi-)norm \cite{BR, WLX}. While $p=1$ is a typical choice, $p=0.5$ has been shown to outperform $p=1$ \cite{ORS}. We therefore address problem \eqref{CS} with $p=0.5$ for varying values of $r$. To apply DCBO, we define the set $S_r = \{ x \in \mathbb{R}^d \mid \|x\|_{0.5} \leq r \}$. Given that $S_r$ is \emph{not convex}, the projection onto $S_r$ is not unique. Thus we modify the objective function $f_r$ to
\[
f_r(x) :=
\begin{cases}
	\frac{1}{2}\|Ax-b\|^2_2 \quad &\text{if} \quad x \in S_r, \\
	+\infty \quad &\text{otherwise}.
\end{cases}
\]

For comparison, we also solve \eqref{CS} using the projected gradient descent (PGD) method \cite{WLX}. We compared PGD and DCBO for $N=200$ and $2000$ using three fixed sparse signals $x_2$, $x_4$, and $x_6$ in $\bbr^{100}$ where $x_s$ is a $s$-sparse vector only $s$ nonzero entries. For each $x_s$, we solved \eqref{CS} using for $r \in \{6,12,18,24,30\}$. In each case, we conducted 100 experiments using a random sensing matrix $A$, where each entry of $A$ is i.i.d. standard normal. The $\| \cdot \|_{0.5}$ norms of each $x_s$ were $\|x_2\|_{0.5} = 3.4655$, $\|x_4\|_{0.5} = 12.9132$ and $\|x_6\|_{0.5} = 21.3583$, and we used $m=40$ measurements. We used $(\gamma^1, \gamma^2, \bar{\gamma}^1, \bar{\gamma}^2) = (0.5,1,0.4,0.7)$ for DCBO. For the PGD implementation, we modified the code from \cite{WLX}.\footnote{The code is available at https://github.com/won-j/LpBallProjection.} To post-process the DCBO and PGD solutions, we applied hard-thresholding to small entries $(<0.01)$ and minimized $\|Ax-b\|_2$ constrained to the support obtained by thresholding. For each $s$ and $r$, we computed the true positive rate (TPR) and false positive rate (FPR).

\begin{table}[h]
\scriptsize
    \begin{tabular}{l|p{1.77cm}|p{1.77cm}|p{1.77cm}|p{1.77cm}|p{1.77cm}}
	\toprule
       \multicolumn{6}{c}{\textbf{PGD}}   \\
        \midrule
        Radius $(r)$ & 6 & 12 & 18 & 24 & 30  \\
        \midrule
        s=2 (TPR) & 0.9200(0.0222) & 0.2950(0.0334) & 0.0850(0.0202) & 0.0300(0.0119) & 0.0200(0.0098)  \\
        s=2 (FPR) & 0.0051(0.0010) & 0.0313(0.0013) & 0.0350(0.0010) & 0.0363(0.0009) & 0.0365(0.0008)  \\
        \midrule
        s=4 (TPR) & 0.5000(0.0036) & 0.8950(0.0225) & 0.3750(0.0317) & 0.0850(0.0164) & 0.0475(0.0122)  \\
        s=4 (FPR) & 0.0002(0.0002) & 0.0049(0.0010) & 0.0285(0.0010) & 0.0340(0.0010) & 0.0355(0.0011)  \\
        \midrule
        s=6 (TPR) & 0.3367(0.0058) & 0.5950(0.0163) & 0.4667(0.0331) & 0.1817(0.0260) & 0.0967(0.0142)  \\
        s=6 (FPR) & 0.0005(0.0003) & 0.0030(0.0008) & 0.0170(0.0015) & 0.0311(0.0010) & 0.0340(0.0011) \\ 
	\toprule
       \multicolumn{6}{c}{\textbf{DCBO $(N=200)$}}   \\
        \midrule
        Radius $(r)$ & 6 & 12 & 18 & 24 & 30  \\
        \midrule
        s=2 (TPR) & 0.8550(0.0278) & 0.9200(0.0184) & 0.9600(0.0169) & 0.9800(0.0121) & 0.9500(0.0167)  \\
        s=2 (FPR) & 0.0093(0.0017) & 0.0130(0.0027) & 0.0077(0.0023) & 0.0200(0.0043) & 0.0388(0.0059)  \\
        \midrule
        s=4 (TPR) & 0.4275(0.0152) & 0.6625(0.0278) & 0.8700(0.0218) & 0.8575(0.0217) & 0.8500(0.0259)  \\
        s=4 (FPR) & 0.0093(0.0012) & 0.0206(0.0017) & 0.0163(0.0023) & 0.0265(0.0033) & 0.0333(0.0040)  \\
        \midrule
        s=6 (TPR) & 0.2933(0.0138) & 0.4250(0.0171) & 0.5733(0.0219) & 0.7383(0.0259) & 0.7683(0.0227)  \\
        s=6 (FPR) & 0.0116(0.0014) & 0.0184(0.0015) & 0.0235(0.0019) & 0.0256(0.0023) & 0.0351(0.0030) \\ 
	\toprule
       \multicolumn{6}{c}{\textbf{DCBO $(N=2000)$}}   \\
        \midrule
        Radius $(r)$ & 6 & 12 & 18 & 24 & 30  \\
        \midrule
        s=2 (TPR) & 0.9350(0.0197) & 0.9750(0.0131) & 0.9700(0.0119) & 1.0000(0.0000) & 0.9900(0.0070)  \\ 
        s=2 (FPR) & 0.0040(0.0012) & 0.0029(0.0014) & 0.0055(0.0022) & 0.0030(0.0013) & 0.0100(0.0032)  \\ 
        \midrule
        s=4 (TPR) & 0.4950(0.0167) & 0.7925(0.0268) & 0.9050(0.0221) & 0.9400(0.0178) & 0.9750(0.0104)  \\ 
        s=4 (FPR) & 0.0055(0.0010) & 0.0130(0.0016) & 0.0095(0.0020) & 0.0074(0.0020) & 0.0060(0.0022)  \\ 
        \midrule
        s=6 (TPR) & 0.3383(0.0126) & 0.5050(0.0168) & 0.6683(0.0209) & 0.8467(0.0235) & 0.9100(0.0210)  \\ 
        s=6 (FPR) & 0.0073(0.0010) & 0.0119(0.0013) & 0.0140(0.0017) & 0.0145(0.0022) & 0.0116(0.0025) \\ 
        \bottomrule
    \end{tabular}
    \caption{Comparison between PGD and DCBO in the compressed sensing problem. Each entry represents mean and standard error. Results were averaged over 100 runs. We solved \eqref{CS} where for the domain parameter $r \in \{6,12,18,24,30\}$.} 
    \label{compTable}
\end{table}

The results in Table \ref{compTable} demonstrate that DCBO exhibited robustness in the radius $r$, tending to successfully capture the zero element whenever $r$ is sufficiently large. In contrast, PGD appears to perform well only within a specific range of $r$. Both algorithms performed poorly when $r$ was too small relative to $\|x_s\|_{0.5}$. Given that neither $s$ nor optimal domain parameter $r$ are known a priori to the algorithms, DCBO offers an advantage of solving \eqref{CS} with minimal dependence on $r$.

\section{Conclusion} \label{sec:6}
This paper introduces the Discrete Consensus-Based Optimization (DCBO) method as a novel multi-agent approach to global optimization. Existing research on discretized CBO has focused on the almost sure consensus or error estimates for the best event under \emph{homogeneous} noise. In our view, this limitation arises from the absence of appropriate tools for discrete time, which complicates the analysis of time-discrete CBO models. Our proposal of replacing the “softmin” action with “hardmin”, which puts an agent with the best function value as the consensus point, liberates the mathematical analysis previously deemed unfeasible, enabling us to extract more comprehensive theoretical results than earlier studies. Further research is warranted, for example, on the direct relationship between the number of iterations and the progression of optimization.




\bibliographystyle{siamplain}

\end{document}